\documentclass[onefignum,onetabnum]{siamart220329}

\usepackage{lipsum}
\usepackage{amsfonts}
\usepackage{graphicx}
\usepackage{epstopdf}
\usepackage{algpseudocode }
\usepackage{booktabs}
\usepackage{amsmath,amssymb}
\usepackage{mathtools}
\usepackage{stmaryrd}
\usepackage{pgfplots}
\usepgfplotslibrary{colorbrewer, groupplots}
\usetikzlibrary{pgfplots.statistics, pgfplots.colorbrewer, external}
\usepackage{pgfplotstable}
\usepackage{filecontents}
\usepackage{longtable}
\usepackage{makecell}

\ifpdf
  \DeclareGraphicsExtensions{.eps,.pdf,.png,.jpg}
\else
  \DeclareGraphicsExtensions{.eps}
\fi

\newcommand{\real}{\mathbb{R}}   
\newcommand{\ex}{\mathbb{E}}   
\newcommand{\pr}{\mathbb{P}}     
\newcommand{\mat}[1]{\boldsymbol{#1}}   
 
\newcommand{\norm}[1]{\left\| {#1} \right\|}

\newcommand{\blkmat}[1]{\begin{bmatrix} #1 \end{bmatrix}}
\newcommand{\twotwo}[4]{\blkmat{#1 & #2 \\ #3 & #4}}
\newcommand{\twoone}[2]{\blkmat{#1 \\ #2}}
\newcommand{\onetwo}[2]{\blkmat{#1 & #2}}
\newcommand{\lowrank}[2]{\left\llbracket #1 \right\rrbracket_{#2}}
\newcommand{\icur}{\cref{alg:ICUR}}
\newcommand{\icurl}{IterativeCUR-LUPP}

\newcommand{\curs}{s-LUPP}
\newcommand{\svds}{SVDSketch}

\newcommand{\fopa}{.3}
\newcommand{\dopa}{1}
\newcommand{\boxplotscale}{.7}

\definecolor{cSVDs}{RGB}{141, 8, 1}
\definecolor{cLU}{RGB}{227, 178, 60}
\definecolor{cIQR}{RGB}{0, 20, 83}
\definecolor{cILU}{RGB}{12, 155, 107}
\definecolor{cOS}{RGB}{251, 97, 7}
\definecolor{cSQR}{RGB}{67, 97, 238}
\definecolor{cSOS}{RGB}{247, 37, 133}
\definecolor{cSVD}{RGB}{27, 99, 110}
\definecolor{cminus}{RGB}{145, 0, 0}
\definecolor{cpos}{RGB}{79, 120, 59}
\definecolor{c5}{RGB}{191,67,66}
\definecolor{c50}{RGB}{231, 215, 193}
\definecolor{c100}{RGB}{167, 138, 127}
\definecolor{c500}{RGB}{115, 87, 81}
\newcommand{\markersize}{1}

\pgfplotsset{
    numformat single/.style={pgf/number format/#1},
    numformat/.style={numformat single/.list={#1}},
    every axis/.append style={
         width = .5 \textwidth,
        height = .4 \textwidth,
        every tick/.style = {xtick pos = left},
        title style={font=\bfseries},
        label style={font=\small},
        ticklabel style={font=\footnotesize},
        legend style={font=\footnotesize},
        scaled y ticks = false,
    },
    every axis plot/.append style={
    every mark/.append style={solid,  fill opacity = .7, draw opacity = 1},
    scur_line/.style = {mark=triangle, mark size = 3pt, color = cLU, dashed},
    icur_line/.style={mark=square, mark size = 2pt, color = cILU, densely dotted},
    svds_line/.style={mark=o, mark size = 2pt, color = cSVDs},
    svd_line/.style={mark=pentagon, mark size = 3pt, color = cSVD, dash dot dot},
    qr_line/.style={mark=diamond, mark size = 3pt, color = cIQR},
    os_line/.style={mark=star, mark size = 3pt, color = cOS, dash dot},
    sos_line/.style={mark=halfcircle, mark size = 1.5pt, color = cSOS},
    sqr_line/.style={mark=-, mark size=3pt, color = cSQR},
    bs5_line/.style={mark=pentagon, mark size = 3pt, color = cSVD, dash dot dot},
    bs50_line/.style={mark=*, mark size = 2pt, color = c500},
    bs100_line/.style={mark=diamond, mark size = 3pt, color = orange, dashed},
    bs500_line/.style={mark=triangle, mark size = 3pt, color = black, densely dotted},
    }
}

\newsiamremark{remark}{Remark}
\newsiamremark{hypothesis}{Hypothesis}
\crefname{hypothesis}{Hypothesis}{Hypotheses}
\newsiamthm{example}{Example}
\crefname{example}{example}{examples}
\newsiamthm{claim}{Claim}

\headers{Iterative CUR}{N. Pritchard, T. Park, Y. Nakatsukasa, and P.G. Martinsson}

\title{ IterativeCUR: Large Rank-Adaptive Approximation from a Small Recycled Sketch \thanks{Submitted to the editors September 26, 2025.
\funding{The work reported was supported by the Office of Naval Research (N00014-
18-1-2354), by the National Science Foundation (DMS-2313434),
and by the Department of Energy ASCR (DE-SC0025312). For the purpose of open access, the authors have applied a CC BY public copyright license to any author accepted manuscript arising from this submission.}}}

\author{Nathaniel Pritchard\thanks{Mathematical Institute at the University of Oxford.}
\and Taejun Park\thanks{Institute of Mathematics at EPFL.}
\and Yuji Nakatsukasa\footnotemark[2]
\and Per-Gunnar Martinsson\thanks{Department of Mathematics and the Oden Institute at the University of Texas at Austin.} 
}

\usepackage{amsopn}

\makeatletter
\newcommand*{\addFileDependency}[1]{
  \typeout{(#1)}
  \@addtofilelist{#1}
  \IfFileExists{#1}{}{\typeout{No file #1.}}
}
\makeatother

\ifpdf
\hypersetup{
  pdftitle={IterativeCUR},
  pdfauthor={Pritchard, Park, Nakatsukasa, Martinsson}
}
\fi

\usepackage{color}

\usepackage[normalem]{ulem}

\newcommand{\ignore}[1]{}

\begin{document}

\maketitle

\begin{abstract}
  The computation of accurate low-rank matrix approximations is central to improving the scalability of various techniques in machine learning, uncertainty quantification, and control. Traditionally, low-rank approximations are constructed using SVD-based approaches such as truncated SVD or Randomized SVD. Although these SVD approaches---especially Randomized SVD---have proven to be very computationally efficient, other low-rank approximation methods can offer even greater performance. One such approach is the CUR decomposition, which forms a low-rank approximation using direct row and column subsets of a matrix. Because CUR uses direct matrix subsets, it is also often better able to preserve native matrix structures like sparsity or non-negativity than SVD-based approaches and can facilitate data interpretation in many contexts. This paper introduces IterativeCUR, which draws on previous work in randomized numerical linear algebra to build a new algorithm that is highly competitive compared to prior work. IterativeCUR is adaptive in the sense that it takes as an input parameter the desired tolerance $\epsilon$ and outputs (with arbitrarily high probability) an approximation of error bounded by $\epsilon$, rather than requiring an a priori guess of the numerical rank.
 IterativeCUR typically runs significantly faster than both existing CUR algorithms and techniques such as Randomized SVD, in particular when these methods are run in an adaptive rank mode. Its asymptotic complexity is  $\mathcal{O}(mn + (m+n)r^2)$ for an $m\times n$ matrix of output rank $r$.
 IterativeCUR relies on a single small sketch from the matrix that is successively downdated as the algorithm proceeds. We demonstrate through extensive experiments that IterativeCUR achieves up to $4\times$ speed-up over state-of-the-art pivoting-on-sketch approaches with no loss of accuracy, and up to $40\times$ speed-up over rank-adaptive randomized SVD approaches.
\end{abstract}


\section{Introduction}\label{sec:int}
For decades, low-rank approximations have been essential for the scalability improvements achieved in a variety of scientific computing fields, 
 ranging from mass spectrometry \cite{yang2015identifying}, surrogate modeling \cite{kramer2024learning,zheng2025semilagrangian}, genetics \cite{mahoney2009cur}, and deep learning \cite{flynn2024stat, mai2020vgg, park2025curing}.  In these fields, practitioners often form low-rank approximations using optimal SVD-based  approaches such as the truncated SVD \cite{eckart1936approximation, mirsky1960symmetric} and the Randomized SVD \cite{halko2011finding}.  However, while some SVD-based methods are very computationally efficient, it turns out that performance can be further improved by switching to alternative low-rank approximation methodologies such as the CUR decomposition, which has the additional advantage that it facilitates data interpretation \cite{mahoney2009cur}. 

The CUR decomposition \cite{goreinov1997theory, mahoney2009cur} aims to find a rank-$r$ approximation to a matrix $\mat A \in \real^{m \times n}$ by choosing columns of $\mat A$, $\mat C \in \real^{m \times r}$, rows of $\mat A$, $\mat R \in \real^{r \times n}$, and a 
core matrix $\mat U \in \real^{r \times r}$ such that 
\begin{equation}
    \mat A \approx \mat C \mat U \mat R.
\end{equation}

The CUR decomposition offers superior scalability and better preservation of sparsity and non-negativity compared to SVD-based methods because it works directly with selected subsets of the original matrix's rows and columns. However, this advantage comes with a trade-off: CUR approximations typically achieve somewhat lower accuracy than their SVD counterparts. This accuracy limitation stems from the constraint that the $\mat C$ and $\mat R$ matrices must consist of actual column and row subsets from the original matrix, preventing CUR from reaching the optimal accuracy that a truncated SVD of the same rank can achieve.
Nonetheless, the CUR factorization will be at worst within a factor of $r+1$ from the optimal SVD-based rank-$r$ approximation in the Frobenius norm, provided that optimal rows and columns are chosen~\cite{cortinovis2020lowrank,deshpande2006no,osinsky2025close}\footnote{This is for the "efficient" choice of $\mat U$, 
$\mat{U}=\mat{A}(I,J)^{\dagger}$, sometimes called CUR cross approximation~\cite{park2025accuracy}. 
One can improve the suboptimality factor to $\sqrt{2(r+1)}$~\cite{cortinovis2020lowrank} with $\mat{U}=\mat{C}^\dagger \mat{AR}^\dagger$ (CUR best approximation), which minimizes the Frobenius norm error; but computing this is expensive, and we make little use of it in this paper. }.

Selecting rows and columns that \textit{optimally} minimize the Frobenius norm of the error is NP-Hard \cite{shitov2021column, civril2009selecting}. Instead, most CUR index selection techniques aim to guarantee high-quality approximations to a matrix. Broadly, these selection techniques can be categorized as sampling techniques, pivoting techniques, or pivoting on a subspace techniques. Sampling techniques select indices by sampling from fixed distributions (e.g., Uniform, Leverage Score, DPP) over the row and column indices \cite{derezinski2021determinantal, frieze2004fast, mahoney2009cur}.  Pivoting techniques select indices by applying a pivoting procedure directly to the matrix (e.g., LU with partial pivoting and QR with column pivoting) \cite{stewart1998matrix}. Pivoting on subspace techniques work by applying a pivoting technique to a subspace of the matrix, typically obtained via the truncated SVD or through matrix sketching. These techniques offer some of the best approximation performance and computational efficiency in practice, but are often more difficult to analyze theoretically \cite{dong2023simpler}. Examples of pivoting on a subspace techniques include pivoting on a sketch approaches such as sketched LU with partial pivoting \cite{dong2023simpler}, DEIM \cite{sorensen2016deim}, Osinsky's selection approach \cite{osinsky2025close} as well as many others \cite{chen2020efficient, cortinovis2024adaptive,duersch2020randomized, osinsky2025close}. Beyond these categories, there are many hybrid approaches, which often involve choosing pivots via sampling that include techniques such as robust block-wise random pivoting \cite{dong2024robust} or RP-Cholesky \cite{chen2025randomly}. 

Typically, index-selection techniques require that the rank be known before selecting the indices \cite{chen2020efficient, cortinovis2024adaptive, dong2023simpler, duersch2020randomized, mahoney2009cur}.  Unfortunately, in many applications, practitioners do not have this knowledge \cite{kramer2024learning, park2025curing, yang2015identifying}. Instead, practitioners often prefer that the approximation, $\widehat{\mat A}$, satisfies a particular quality threshold, for example, the relative error, $\|\mat A - \widehat {\mat A}\|_F / \|\mat A\|_F < \epsilon$. Current CUR approaches can approximate this behavior using a rank estimation procedure such as the one proposed in \cite{meier2024fast} that uses estimates of singular values to determine a rank that corresponds to the desired quality threshold. Unfortunately, owing to the sub-optimality of the CUR for the best rank-$r$ approximation to a matrix, this approach does not guarantee that the returned CUR approximation will satisfy the desired quality threshold.

Guaranteeing quality requires a rank-adaptive CUR approach. Rank-adaptive approaches iteratively increase the rank of an approximation until a particular quality threshold is satisfied. As an example, a rank-adaptive variant of RandomizedSVD \cite{halko2011finding} known as SVDSketch \cite{yu2018efficient} iteratively updates the range approximator, $\mat Q$, with sketches of size $b$ of $\mat A$ until $\|\mat A\|_F^2 - \|\mat Q^\top\mat A\|_F^2$ is small enough. Beyond Randomized SVD, several innovative and highly effective rank-adaptive approaches have been proposed for one-sided interpolatory decompositions, which are closely related to CUR but select \textit{either} rows \textit{or} columns (not both) \cite{dong2024robust, pearce2025adaptive}.

This work introduces a rank-adaptive CUR framework, IterativeCUR, that requires only $\mathcal{O}(r(m+n))$ memory to update the approximation with the same computational complexity, $\mathcal{O}(mn+(m+n)r^2)$, as other state-of-the-art CUR approaches \cite{dong2023simpler}. IterativeCUR forms a CUR approximation by selecting indices from a single recycled sketch of the residual, $\mat A - \mat C \mat U\mat R$, that holds $\mathcal{O}(1)$ vectors. This single sketch allows IterativeCUR to thrive in environments where matrix-vector operations are expensive, e.g., when the matrix $\mat A$ is very large and dense. The sketch of the residual also facilitates the creation of a risk-aware stopping criterion for IterativeCUR, which allows users to control the probability that the approximation returned by IterativeCUR violates their quality threshold. 
 In experiments, we demonstrate the effectiveness of IterativeCUR on a range of moderately sized real-world and synthetic matrices.  Specifically, we observed that IterativeCUR is four times faster than the state-of-the-art column selection procedure, Sketched LUPP \cite{dong2023simpler} (which is not rank-adaptive), while matching its accuracy. Additionally, compared to the rank-adaptive approach for the Randomized SVD, SVDSketch, IterativeCUR is typically nearly as accurate and up to 40 times faster.  Experiments examining IterativeCUR revealed that high accuracy is maintained even for block sizes as small as five. However, for optimal computational speed, larger block sizes (around 100) perform better. The experiments further indicate that while the choice of index selection strategy does have an impact on accuracy, the differences are small. For more details on the experiments, see \cref{sec:exp}. 

Put simply, this paper introduces an efficient and practical method for computing CUR approximations that substantially outperforms Randomized SVD. The approach achieves this efficiency through two key features: it requires surprisingly few matrix-vector products with the full matrix (as few as five), and its rank-adaptive design automatically determines a sufficient rank to achieve  $\epsilon$-approximations with  $\epsilon$ approaching machine precision. By contrast, SVDSketch is unable to ensure the accuracy of an approximation for any accuracy higher than the square root of the machine precision~\cite{yu2018efficient}.

In the next subsection, we introduce the relevant notation for the paper.  After this notation is introduced, the paper will proceed in the following manner. 
In \cref{sec:pre}, we present background relating to the CUR decomposition and random embeddings that is relevant for understanding IterativeCUR.
In \cref{sec:alg}, we present the intuition, describe the specifics, and explore the computational complexity of IterativeCUR. 
In \cref{sec:the}, we present an upper bound on the accuracy of IterativeCUR.
In \cref{sec:exp}, we present extensive numerical experiments that demonstrate the efficiency of row and column selection compared to other fixed-rank approaches, the overall effectiveness of IterativeCUR when identifying an approximation for a specified error tolerance, and the impact of sketch size and index selection method on the performance of IterativeCUR. Finally, in \cref{sec:con}, we conclude with a summary of IterativeCUR's practical advantages and remaining theoretical challenges.

\subsection{Notation}\label{sec:not}
This paper will focus on techniques to construct a rank-$r$ approximation to a matrix $\mat A$. To refer to the $j^{\text{th}}$ column of $\mat A$ we use the notation $\mat A(:, j)$ and to refer to the $i^{\text{th}}$  row of $\mat A$ we use the notation $\mat A(i,:)$. 

Our approach will iteratively form a $\mat C \mat U \mat R$ approximation to $\mat A$ where $\mat C = \mat A(:, J)$, $\mat U = \mat A(I, J)^\dagger$, and $\mat R = \mat A(I, :)$. Here, $I$ is the set of all row indices selected after $k$ iterations, and $J$ represents the set of all column indices selected after $k$ iterations.  Each iteration improves the approximation of iteration $k-1$ by selecting $b=O(1)$ new row indices $I_k$ and column indices $J_k$. We will denote appending the new indices $I_k$ and $J_k$ to $I$ and $J$ with $I = [I, I_k]$ and $J = [J, J_k]$.

The true approximation quality of a CUR approximation can be determined by computing its residual, $\mat S = \mat A- \mat C \mat U \mat R$. In addition to this quantity, we also define two other residuals. The first is the column-residual, $\mat S_{k}^{\rm col}$, which applies a fixed sketching matrix $\mat G \in \real^{c \times m}$, $c = \lfloor 1.1b \rfloor$ where $\lfloor \cdot \rfloor$ is the floor function to the residual to obtain $\mat S_{k}^{\rm col} = \mat G \mat A - \mat G \mat C \mat U \mat R$. The second quantity is the row-residual, $\mat S_{k}^{\rm row}$,  which is the residual at the column indices $J_k$. Specifically, $\mat S_{k}^{\rm row} = \mat A(:, J_k) - \mat C \mat U \mat R(:, J_k)$. We can then use the column residual to measure the quality of a CUR approximation $\rho_k = \|\mat S_{k}^{\rm col}\|_F/\|\mat G \mat A\|_F$. We can also determine that a solution is good enough when $\rho_k < \epsilon$, where $\epsilon$ is a user-selected criterion for approximation quality.

\section{Preliminaries}\label{sec:pre}

The following section will present the background necessary to describe our proposed method for forming CUR approximations. We will begin by developing an understanding of the power of random vectors for revealing dominant singular subspaces through a discussion of the random subspace embeddings and the Randomized Rangefinder. With an understanding of the power of random sketching, we introduce important ideas relating to CUR approximations.

\subsection{Subspace Embeddings and the Randomized Rangefinder}
A powerful tool in the field of Randomized Numerical Linear Algebra is the random subspace embedding. We define a random subspace embedding as 
\begin{definition}\label{def:subspace}
    For a given matrix $\mat A \in \real^{m \times n}$, distortion parameter $\delta \in (0,1)$, and probability $\gamma \in (0,1)$, a random linear map $\mat G: \real^{m} \rightarrow \real^b$ is considered a {random subspace embedding with distortion $\delta$} if for all $x$
    \begin{equation}
        \pr \left((1-\delta) \|\mat Ax\|_2 \leq \|\mat G\mat Ax\|_2 \leq (1+\delta)\|\mat Ax\|_2\right) \geq 1- \gamma. 
    \end{equation}
\end{definition}
Random linear maps that satisfy \cref{def:subspace} for large enough $b$ include Gaussian matrices \cite{indyk1998approximate}, sub-sampled random trigonometric transforms \cite{ailon2009fast,tropp2011improved,woolfe2008fast}, and the sparse sign matrix \cite{cohen2016nearly, meng2013lowdistortion}. Because subspace embeddings approximately preserve the norms of matrices, they allow
linear algebra operations to be performed on a substantially lower-dimensional space. The reductions in computations afforded by operating on this lower-dimensional space have been central to the performance improvements offered by Randomized Numerical Linear Algebra approaches.

The Randomized Rangefinder \cite{halko2011finding} is one technique that uses random linear maps with the subspace embedding property (\cref{def:subspace}) to form accurate low-rank approximations to the range of a matrix $\mat A$. Remarkably, the Randomized Rangefinder can accurately form these low-rank approximations even if the embedding dimension, $b$, is not large enough for the chosen random linear map to satisfy \cref{def:subspace}.  Specifically, noting that \cite{eckart1936approximation} shows that the optimal range approximator to a matrix $\mat A$ is given by the $r$ right singular vectors, $\mat V_r$, of a matrix, $\mat A$, corresponding to the largest singular values of $\mat A$ and has an error $\|\mat A - \mat A\mat V_r \mat V_r^\top\|_F^2 = \sum_{i=r+1}^{\min(m,n)}\sigma_i(\mat A)^2$, \cite{halko2011finding} shows that if $\mat X = \mat G \mat A$ with $\mat{G} \in \real^{b \times m}$ 
then 
    \begin{equation}\label{thm:rangefinder}
     \ex \left[ \|\mat A - \mat A \mat{X}^\dagger \mat{X}\|_F^2 \right] \leq \left(1+\frac{r}{b-r-1}\right) \|\mat A - \mat A\mat V_r \mat V_r^\top\|_F^2.
    \end{equation} 

This result demonstrates that for $b$ slightly larger than $r$, multiplying $b$-Gaussian vectors with a matrix $\mat A$ is remarkably effective at approximating the dominant $r$-dimensional space. As noted in the original work, the quality of this approximation depends on the decay of the singular values \cite{halko2011finding}. In particular, when the singular values decay slowly, the bound is tight. Conversely, when the singular values decay quickly, the approximation can be significantly more accurate than \cref{thm:rangefinder} would indicate. 

\subsection{The CUR Decomposition}

The CUR decomposition is a low-rank approximation technique that approximates the matrix $\mat{ A} \in \real^{m \times n}$ by selecting subsets of the rows and columns of $\mat A $. In notation, we represent the CUR decomposition as selecting row indices $I$ and column indices $J$ and then forming the decomposition  
\begin{equation}
    \mat A \approx  \mat C \mat U \mat R,
\end{equation}
where $\mat C = \mat A(:, J)$, $\mat R = \mat A(I,:)$, and $\mat{U}$ is a small weight matrix. The theoretically optimal choice for the weight matrix is $\mat U = \mat C^\dagger \mat A \mat R^\dagger$, but this can be expensive to form. A more economical choice that often provides comparable accuracy is the  ``cross approximation'', $\mat U = \mat A(I, J)^\dagger$. We note that when the cross-approximation is used, it is often helpful to select indices dependently for purposes of numerical stability. That is, first select rows (or columns), and then select columns (or rows) based on the information in the chosen rows (columns). See \cite{park2025accuracy} for a full discussion. 

Aside from the index selection approaches discussed in \cref{sec:int}, another important consideration for CUR methods is rank-adaptivity. Unlike standard techniques that require a priori knowledge of the rank, rank-adaptive techniques iteratively increase the rank of an approximation until the approximation satisfies a specified error tolerance, $\epsilon$. For rank-adaptive techniques to be effective, the error of the approximation must be at least cheap to estimate. For example, in the case of Randomized SVD an effective rank-adaptive approach relies on the observation that the squared Frobenius norm of the error is the difference between the squared Frobenius norm of the original matrix and the squared Frobenius norm of its approximation \cite{yu2018efficient}. In the case of CUR, no such equivalence exists, and thus rank-adaptivity is typically a harder goal to achieve. 

Progress has been made in developing rank-adaptive approaches for interpolative decompositions, which can be thought of as one-sided CUR decompositions where only row or column indices are selected. Rank-adaptive interpolative approaches include the sketched LU approach of Pearce et al. \cite{pearce2025adaptive} that relies on the sketched residual to attain rank-adaptivity in a similar manner to IterativeCUR and a closely related hybrid sampling approach known as Robust Blockwise Randomized Pivoting of Dong et al. \cite{dong2024robust}. Unfortunately, neither approach easily generalizes to CUR because of CUR's need for selecting both rows and columns dependently at each iteration. 

We note that there is a strong connection between CUR and LU with complete pivoting. Namely, 
let $P_I,P_J$ be permutation matrices that bring the indices $I,J$ to the leading ones (the ordering of the other indices does not matter). $A\approx CUR$ gives the same approximation as $P_IAP_J \approx LU$, the partial LU factorization obtained by running $r$ steps of the LU elimination, that is, $L\in\mathbb{R}^{m\times r},U\in\mathbb{R}^{r\times n}$. 
That is, we have $CUR=P_I^TLUP_J^T$. As a result, one can use IterativeCUR to compute an LU approximation with complete pivoting to any matrix $A$.

\section{IterativeCUR}\label{sec:alg}
 With an understanding of the ideas and techniques underlying the CUR decomposition, we now lay out the intuition and details for IterativeCUR. This discussion will first overview the IterativeCUR algorithm (\icur). After overviewing IterativeCUR, we will discuss important aspects of IterativeCUR. In particular, we will discuss how IterativeCUR selects row and column indices, which will reveal how IterativeCUR  iteratively constructs the CUR decomposition. With this intuition for constructing the decomposition, we will then discuss stopping IterativeCUR via a computationally efficient a posteriori estimator and stopping criterion. We then finish this section by totaling the computational complexity of IterativeCUR.

\subsection{Overview of IterativeCUR}
When the rank of a matrix is unknown, we usually want to form a CUR approximation that satisfies $\|\mat A - \mat C \mat U \mat R\|_F / \|\mat A\|_F < \epsilon$. A na\"ive approach for creating an $\epsilon$-accurate CUR approximation selects a set of $b$ column indices, $J$, a set of $b$ row indices, $I$, and computes the residual of the CUR decomposition,
\begin{equation}
    \mat S = \mat A - \mat C \mat U \mat R.
\end{equation}
From this residual, one then selects $b$ additional column indices, $J_+$, and $b$ additional row indices, $I_+$. These indices are then appended to the previously selected sets so that $J = [J, J_+]$ and $I = [I, I_+]$ after which a new residual is computed. This process is repeated until $\|\mat A -\mat C \mat U \mat R\|_F/ \|\mat A\|_F < \epsilon$. Although this na\"ive approach can be performed for small matrices, for large matrices, it becomes computationally impractical owing to the $\mathcal{O}(mnr)$ computational complexity and the expense of having to access the full matrix $\mat A$ at each iteration. 

At a high level, IterativeCUR (\cref{alg:ICUR}) performs the na\"ive approach, using a single sketch of $\mat S$, $\mat G \mat S \in \real^{c \times n}$ with $m \gg c > b$ and a block subset of the columns of $\mat S$. Specifically, at each iteration, IterativeCUR selects a new block of column indices, $J_+$, by applying an index selection procedure, such as partially pivoted LU (LUPP), to $\mat G \mat S$. Then once $J_+$, is selected, IterativeCUR selects a new set of row indices, $I_+$, by applying an index selection procedure to the matrix $\mat S(:, J_+)$.  Once the new row and column indices are selected, IterativeCUR updates the matrices $\mat C, \mat U, \mat R$, by computing a new residual $\mat G\mat S$, and repeating this procedure until $\|\mat G \mat S\|_F / \| \mat G \mat A\|_F < \epsilon$. 

In total, for a dense matrix, IterativeCUR targets an $\epsilon$-accurate approximation to a matrix $\mat A$ in $\mathcal{O}(mn + (m+n)r^2)$ operations where $r$ is the rank of the $\epsilon$-accurate approximation.
This computational complexity is comparable to other efficient rank-adaptive randomized SVD algorithms, e.g. SVDSketch \cite{yu2018efficient}. However, unlike SVDSketch, which uses $\|\mat A - \widehat{\mat{A}}\|_F^2$ to make stopping decisions and can only provide approximations up to $10^{-8}$ when using double precision arithmetic, IterativeCUR makes stopping decisions using $\|\mat{GA} - \mat{GCUR}\|_F$ allowing it to generate approximations that are accurate to near machine precision.

\begin{algorithm}[h]
\caption{IterativeCUR}
        \hspace*{\algorithmicindent } \textbf{Input:}  \\ \hspace*{\algorithmicindent }  \quad $\epsilon > 0$ \Comment{The desired quality of $\|\mat A - \mat C \mat U \mat R\|/\|\mat A\|_F$}\\  \hspace*{\algorithmicindent } \quad $b > 0$ \Comment{ Block/sketch size; usually $5\leq b\leq 250$
        }\\  \hspace*{\algorithmicindent} \quad $\textbf{Column Selection}$ \Comment{LUPP as default could use QRCP or Osinsky \cite{osinsky2025close}} \\ \hspace*{\algorithmicindent} \quad $\textbf{Row Selection}$ \Comment{LUPP as default could use QRCP or Osinsky \cite{osinsky2025close}}\\
        \hspace*{\algorithmicindent} \textbf{Output:} $\mat C$, $\mat U$, $\mat R$
    \label{alg:ICUR}
    \begin{algorithmic}[1]
        \Procedure{IterativeCUR}{$b$, $\epsilon$, \textbf{Column Selection}, \textbf{Row Selection}}
        \State Generate a sketch matrix $ \mat G \in \mathbb{R}^{\lfloor 1.1b \rfloor \times m}$
        \State $J = \emptyset$, $I = \emptyset$, $\mat U_{0} = \emptyset$, $\mat C_{0} = \emptyset$, $\mat R_{0}=\emptyset$
        \State $ \mat S^{\rm col}_0 = \mat G \mat A$ \label{alg:initial-sketch}
        \State $\rho_0 = \|\mat S^{\rm col}_0\|_F / \|\mat{GA}\|_F$ 
        \State $k = 1$
         \While{$\rho_{k-1} > \epsilon$}
           \State $J_k \leftarrow \textbf{Column Selection} (\mat S^{\rm col}_{k-1})$
           \State $\mat  S^{\rm row}_k = \mat A(:,J_k) - \mat C_{k-1} \mat U_{k-1} \mat R_{k-1}(:, J_k)$ \label{alg:right-res} \Comment{Compute row residual}
           \State $J \leftarrow [J, J_k]$
           \State $I_k \leftarrow \textbf{Row Selection}(\mat S^{\rm row}_k)$
           \State $I \leftarrow [I, I_k]$
           \State $\mat U_k \leftarrow \mat A(I,J)^\dagger$ \label{alg:update}\Comment{Default QR, improved efficiency with LU}  
           \State $\mat C_k \leftarrow \mat A(:,J)$,  $\mat R_k \leftarrow \mat A(I,:)$ 
           \State $\mat S^{\rm col}_k \leftarrow \mat G \mat A -\mat G \mat C_k \mat U_k \mat R_k$ \Comment{Use $\mat {GC}_k$ from $\mat{GA}$}
           \label{alg:left-res}
           \State $\rho_k \leftarrow \|\mat  S^{\rm col}_k\|_F /\|\mat G\mat A\|_F$\label{alg:left-res-norm} 
           \State $k \leftarrow k+1$
        \EndWhile   
    \EndProcedure
    \end{algorithmic}
\end{algorithm}

\subsection{Index Selection}
With an overview of IterativeCUR, we now examine the details of the algorithm. We begin with its approach to selecting row and column indices. For this discussion, we rely on the observation by \cite{park2025accuracy} that when using the cross-approximation form of $\mat U$, $\mat A(I,J)^\dagger$, row and column indices should be selected in a dependent manner to ensure the numerical stability of the approximation. 
That is, if we first select a set of column indices, $J$, we should choose row indices using only the selected columns, i.e. $\mat A(:,J)$, or we risk the core matrix being poorly conditioned. For example, when performing CUR on the matrix 
\begin{equation*}
    A = 
    \begin{bmatrix}
    0 & \frac{1}{\sqrt{n}} & \dots & \frac{1}{\sqrt{n}} \\
        \frac{1}{\sqrt{n}}+\eta & 0 & \dots & 0
    \end{bmatrix} \in \real^{2\times (n+1)}
\end{equation*}
we can apply QR with Column Pivoting (QRCP) to the rows and columns independently to form a rank-one approximation. If $\eta + 1 / \sqrt{n} < 1$, selecting rows and columns independently will result in us selecting the first row and first column of the matrix, meaning the cross approximation core will be $1/0$. On the other hand, if instead we were to select entries dependently by selecting columns and then rows, we would select the first column and second row, resulting in $1 / (\eta + 1/\sqrt{n})$ being the core matrix.

With an understanding of the need for dependent index selection, in the remainder of this subsection, we present intuition for how to select column indices using $\mathcal{O}(bn)$ memory. Then we present how to use the selected columns to select row indices using $\mathcal{O}(bm)$ memory.

\subsubsection{Column Selection}
To ensure that selecting column indices at each iteration of IterativeCUR incurs the same, small, constant computational and storage cost, we select column indices from the sketched residual matrix, $\mat G(\mat A- \mat C\mat U\mat R)$ where $\mat G \in \real^{c \times m}$, $m \gg c \geq b$, is a random matrix with Gaussian entries. Authors have repeatedly shown that selecting indices from such a sketch of a matrix is highly effective \cite{chen2020efficient, duersch2020randomized, dong2023simpler,  martinsson2017householder, melgaard2015gaussian, pearce2025adaptive}. Typically, these authors select these indices using QR with column pivoting (QRCP) or LU with partial pivoting (LUPP).\footnote{There is no reason in principle why one could not also use sampling techniques such as leverage score sampling to perform the selection \cite{mahoney2009cur}.} Interestingly, applying LUPP or QRCP to a sketch of $\mat A$ typically results in better index selection than applying either technique to $\mat A$ directly \cite{dong2024robust, duersch2020randomized}. 

The success of pivoting on a sketch of $\mat A$ over pivoting directly on $\mat A$ may at first seem surprising. However, the success becomes clearer when we consider that, as \cref{thm:rangefinder} demonstrates, sketching $\mat A$ results in a representation that approximately captures the dominant $b$-dimensional subspace of $\mat A$. Therefore, when we pivot on a sketch of $\mat A$, we select columns that best align with an approximation to the subspace most important for selecting the indices that lead to the most accurate approximation of $\mat A$ \cite{cortinovis2024adaptive, osinsky2025close}. In this way, approaches that pivot on a sketch behave more similarly to approaches that use the true best $b$-dimensional approximation to $\mat A$ like DEIM \cite{chaturantabut2010nonlinear,sorensen2016deim} or Osinsky \cite{osinsky2025close}.

Considering that much like \cite{pearce2025adaptive}, IterativeCUR operates on the residual of the previous iteration's CUR approximation, sketching at each iteration reveals an approximation to the best $b$-dimensional subspace of that residual matrix.  This means that when IterativeCUR uses a pivoting approach to select indices from the sketched residual, the selected indices will closely align with the dominant directions left unexplained by the previous iteration's CUR approximation.

This intuition of index selection leads to two distinct implications for the nature of IterativeCUR. First, because IterativeCUR makes selection decisions with only $b$-dimensional information rather than the ideal $\min\{m,n\}$-dimensional information, IterativeCUR is a greedy algorithm.  Second, by selecting the indices that most align with the dominant directions of the residual, IterativeCUR is able to select indices that correct for a previous iteration's poorly selected indices, if any. In other words, IterativeCUR is a self-correcting algorithm. As we will see in \cref{sec:exp}, these two implied behaviors mean that IterativeCUR is aware of the mistakes caused by its greed and able to correct for them, which leads it to exhibit what we term a \textit{conscientiously greedy} behavior. 

Aside from the selection method, another important aspect of IterativeCUR's column selection is the reuse of a single $\mathcal{O}(1)$ sketch. Computationally, the benefits of reusing a sketch are obvious: we only need to access the full matrix once for $b$ matrix-vector products, in line \ref{alg:initial-sketch} of \cref{alg:ICUR}.  At the same time, the potential drawback is that the dependence induced by the reuse of the sketch could lead to a systematic avoidance of important vectors. Fortunately, other authors who have considered reusing a sketch in pivoting on a sketch procedures such as \cite{chen2020efficient, martinsson2017householder} have not observed this to be the case. This is remarkable when we note that the residual of Randomized SVD $\mat{A}-\mat{AVV}^T$ is orthogonal to the sketch, i.e., $\mat{G}(\mat{A}-\mat{AVV}^T)=\mat{0}$, by construction. Here, $(\mat{GA})^T=\mat{VR}$ is the economical QR factorization. 

To see why dependency is not nearly as problematic as we might expect in CUR, we consider the following contrived example.

\begin{example}
    Let $\mat A = \begin{bmatrix}
        c_1 v_1 & c_2 v_2 & c_3 v_3
    \end{bmatrix} \in \real^{3\times 3}$ where $c_1 > c_2> c_3 > 0$ and $v_1,v_2,v_3$ orthonormal. Additionally, assume that we generate a single sketching vector $g = d_2 v_2 + d_3 v_3$, where $d_2 > d_3 >0$. Finally, for simplicity, assume that $v_2(2) > v_2(1) >v_2(3) >0$ and that $v_1(2) = v_3(2) > 0$. Now, let us consider using IterativeCUR to form a rank-2 approximation of the matrix $\mat A$. It should be clear that the best approximation includes the first and second columns of $\mat A$. We now see what happens when we apply IterativeCUR reusing the sketching vector $g$, which is pathological in that it is orthogonal to the first column of $\mat A$.
    \vspace{1em}

    Iteration 1: Applying the sketching vector $g$ to $\mat A$ we obtain:
    \begin{equation}
        g^\top \mat A = \begin{bmatrix}
            0 & c_2 d_2 & c_3 d_3
        \end{bmatrix}.
    \end{equation}
    Taking the norms of the columns of $g^\top \mat A$, it is obvious based on the relationships between $c_2,d_2$ and $c_3, d_3$ that we would select the second column. Using our knowledge about the entries of $v_2$ we know that the second entry has the largest row and thus we select the second row. 
    \vspace{1em}

    Iteration 2:
    Selecting the second column and second row of $\mat A$ means that the residual after the first iteration is:
    \begin{align}
        g^\top & (\mat A - \mat C \mat U \mat R) \\
        &= g^\top \left(\begin{bmatrix}
        c_1 v_1 & c_2 v_2 & c_3 v_3
        \end{bmatrix} - \frac{c_2 v_2}{c_2 v_2(2)}  \begin{bmatrix}
            c_1v_1(2) & c_2v_2(2) & c_3v_3(2)
        \end{bmatrix}\right)\\
        &=\begin{bmatrix}
            -\frac{c_1 d_2}{v_2(2)} v_1(2) & 0& c_3 d_3 - \frac{c_3 d_2}{v_2(2)} v_3(2)
        \end{bmatrix}
    \end{align}
    Using the constant relations defined at the beginning of the example, we select the first column and obtain the optimal rank-2 approximation of the matrix $\mat A$. 
\end{example}

In this example, we see how selecting from the residual mitigates the dependency of index selection that comes from reusing a single sketching vector. Superficially, this example's sketching vector is bad because it has no components aligned with the most important vector. However, using the residual mitigates much of the adverse effect of the reused sketch because the sketched residual punishes columns most aligned with the second column and rewards those that are both large and not aligned with the second column. This results in IterativeCUR still selecting the first column in the second iteration, as we would hope. Of course, this mitigation is not perfect. If we were to allow $v_1(2) = v_3(2) = 0$, then we would select incorrect indices. However, when we consider that we typically use a sketch size greater than one and that a Gaussian vector is orthogonal to any particular vector with probability zero, then we can see how, provided that we are selecting enough indices, it is difficult to construct situations where,  even when reusing the sketch, IterativeCUR will select many bad indices. 

 \subsubsection{Row Selection}
 Because IterativeCUR uses the cross-approximation as its core matrix, it is important that entries at the intersection of the selected rows and selected columns form a non-singular matrix. Superficially, there are two ways we can ensure that we select rows that interact well with the previously selected columns. We can either choose the rows by pivoting on the residual at \textbf{all} previously selected columns or we can choose rows by looking at the residual at only \textbf{the most recently} selected columns. It turns out that, if we look carefully at the residual of the CUR decomposition, in the context of IterativeCUR, these choices are equivalent. First, if we let $I$ be the set of selected row indices and $J$ be the set of selected column indices. We define $I^c = \{k: k \in [1, \dots, m] \cap k \not\in I\}$ and $J^c = \{k: k \in [1, \dots, n] \cap k \not\in J\}$. We let $\mat C = [\mat A(I, J)^\top\,\, \mat A(I^c,J)^\top]^\top$, $\mat R = [\mat A(I, J)\,\, \mat A(I,J^c)]$, and $\mat U = \mat A(I,J)^\dagger$
 (which here we assume is square and nonsingular, as would almost always be the case in IterativeCUR), for an appropriately pivoted $\mat A$ the residual of a particular CUR, $\mat{S} = \mat{A} - \mat{C} \mat{U} \mat{R}$, is

\begin{align*}
    \mat{S} = \begin{bmatrix}
        \mat{A}(I,J) & \mat{A}(I, J^c)\\
        \mat{A}(I^c, J) & \mat{A}(I^c, J^c)
    \end{bmatrix} - 
    &\begin{bmatrix}
        \mat{A}(I, J)\\
        \mat{A}(I^c,J)
    \end{bmatrix}
    \begin{bmatrix}
        \mat{A}(I,J)
    \end{bmatrix}^{\dagger}
    \begin{bmatrix}
        \mat{A}(I,J) & \mat{A}(I,J^c)
    \end{bmatrix}\\ &= 
    \begin{bmatrix}
        \mat 0 & \mat 0\\
        \mat 0 &  \mat{A}(I^c, J^c) -  \mat{A}(I^c,J)  \mat{A}(I,J)^{-1} \mat{A}(I, J^c)
    \end{bmatrix}.
\end{align*}
Looking at the residual matrix $\mat S$, it is clear that the residual entries corresponding to previously selected rows, $I$, and columns, $J$, are $\mat 0$. Thus, for a new selected set of column indices $J_+$, the row indices that best align with the selected columns $J \cup J_+$ can only be determined using $\mat S(:,J\cup J_+)$. Since $\mat S(:,J) = \mat 0$, selecting rows from $\mat S(:,J\cup J_+)$ is equivalent to selecting rows from $\mat S(:, J_+)$. This reduction allows for row selections to be performed with similar computational complexities and memory costs to those incurred while selecting columns.

\subsection{Stopping IterativeCUR} \label{subsec:stopping}
With an understanding of how IterativeCUR updates its low-rank approximation, we now discuss how to cheaply determine when an $\epsilon$-accurate solution has been achieved.  Because of the dependence of the residual computations, it is difficult to obtain strong guarantees on the accuracy of the error estimate. We therefore consider the case where an appropriately scaled Gaussian matrix is used as the sketch, and $\|\mat A\|_F$ is used in place of $\|\mat G\mat A\|_F$ in the denominator of the relative residual estimate to derive a good stopping criterion. Additionally, to remove dependencies between the selection and error evaluation process we 
can take extra $w=O(1)$ vectors in the sketch, which are strictly separated for use in the error evaluation, and not part of the algorithm. 
All of these are 
simple variants of \cref{alg:ICUR}, and are made to derive a stopping rule that can act as a reasonable heuristic to provide rigorous control over the stopping behavior, if desired.

In an ideal world, we would be able to use the relative Frobenius norm of the residual as an estimator, $\|\mat A - \mat C \mat U \mat R\|_F / \|\mat A\|_F$, and use that to make exact stopping decisions. Of course, computing this quantity requires that we access the full $\mat A$ at each iteration, which is expensive. To avoid this computational expense, we use the Frobenius norm of the sketched residual as an estimate of this quantity (see line \ref{alg:left-res-norm}). 

 Although careful analysis by \cite{pearce2025adaptive} has shown that the norm of the sketched residual is an extremely accurate estimator, it is still reasonable to want to control the amount that the norm of the sketched residual deviates from the norm of the full residual. To come up with a criterion that provides the user with such control when stopping IterativeCUR, we rely on the probabilistic bound derived by \cite{gratton2018improved} when the sketching matrix is Gaussian. Specifically, the bound is 
\begin{theorem}[Theorem 3.1 \cite{gratton2018improved}] \label{thm:gratton-bounds}
    Let $\mat A \in \real^{m \times n}$ be a matrix of rank $r$ whose Frobenius norm is to be estimated, and let $\mat G \in \real^{{w} \times m}$ be a matrix with independent mean zero variance $1 / {w}$ Gaussian entries. Then for any $\tau > 1$
    \begin{equation}
        \pr\left[ \|\mat G \mat A\|_F \leq (1/\tau) \|\mat A\|_F\right] \leq \exp\left( -\frac{{w}(\tau ^2- 1)^2}{4 \tau^4}\right).
    \end{equation}
\end{theorem}

Using the bound from \cite{gratton2018improved}, we wish to determine an adjustment to $\epsilon$ that will provide high-probability guarantees that the deviation of $\|\mat G\mat A\|_F$ from $\|\mat A\|_F$ does not result in IterativeCUR stopping before the desired quality of approximation has been reached. To do this, we allow the user to specify $\epsilon$, $\delta >0 $ representing the maximum allowable deviation $\|\mat G \mat A\|_F$ from $\|\mat A\|_F$, and $\alpha \in (0,1)$ representing the probability that the $\delta$ deviation is not satisfied. Then combining this information with the bound of \cite{gratton2018improved}, we can derive the following corollary.

\begin{corollary}\label{cor:stopping}
    Given  $\epsilon, \delta >0$, and selecting $\alpha\in(0,1), {w}>0$ such that $w>-4\ln(\alpha)$, then for matrices $A \in \real^{m \times n}$ and $\mat G \in \real^{{w} \times m}$ where $\mat G$ has independent mean zero, variance $1 / {w}$ Gaussian entries if  $\xi =  (1+\delta)\sqrt{1-2\sqrt{-\ln(\alpha)/{w}}} $ we have 

   \begin{equation}\label{eq:cor-prob}
       \pr\left[\frac{\|\mat G(\mat A-\mat{CUR})\|_F}{\|\mat{A}\|_F} < \xi \epsilon \ \bigcap \ \frac{\|\mat A-\mat{CUR}\|_F}{\|\mat A\|_F} > (1+\delta)\epsilon \right] \leq \alpha.
   \end{equation}
\end{corollary}
\begin{proof}
We begin by observing that we can bound \cref{eq:cor-prob}
    \begin{align}
        \notag \pr&\left[\frac{\|\mat G(\mat A-\mat{CUR})\|_F}{\|\mat A\|_F} < \xi \epsilon \ \bigcap \ \frac{\|\mat A-\mat{CUR}\|_F}{\|\mat A\|_F} > (1+\delta)\epsilon \right]\\
        \notag \\
        & \notag =\pr\left[\frac{\|\mat G(\mat A-\mat{CUR})\|_F\|\mat A-\mat{CUR}\|_F}{\|\mat A-\mat{CUR}\|_F\|\mat A\|_F} < \xi \epsilon \ \bigcap\  \frac{\|\mat A-\mat{CUR}\|_F}{\|\mat A\|_F} > (1+\delta)\epsilon \right]\\
        \notag \\
        & \notag  \leq \pr\left[\frac{\|\mat G(\mat A-\mat{CUR})\|_F}{\|\mat A-\mat{CUR}\|_F} < \xi/(1+\delta)\right],
    \end{align}
    where the final inequality comes from using the lower bound of $\|\mat A-\mat{CUR}\|_F / \|\mat A\|_F$. Now by letting $s = 1/\tau = \xi/(1+\delta)$ 
    applying the bound from \cref{thm:gratton-bounds} we have
    \begin{equation}
         \pr\left[\frac{\|\mat G(\mat A-\mat{CUR})\|_F}{\|\mat A-\mat{CUR}\|_F} < \xi/(1+\delta)\right] \leq \exp\left( -\frac{{w}s^4 (1- 1/s ^2)^2}{4}\right)
    \end{equation}
    Setting the right-hand side equal to $\alpha$ and solving the resulting equation for $s$ gives the desired result.
\end{proof}
 This bound tells us that to prevent a problematic failure, a failure where $ \|\mat{A} - \mat{CUR}\|_F / \|\mat A\|_F> (1+\delta) \epsilon$, we should stop IterativeCUR using the threshold $\epsilon(1+\delta)\sqrt{1-2\sqrt{-\ln(\alpha)/{w}}}$. {Additionally, because we are making stopping decisions at each iteration, we still have multiple testing issues \cite{tian2021online}. One can alleviate these concerns by using $\alpha_i$ in place of $\alpha$ in the equation for $\xi$ in \cref{cor:stopping} where for the i$^\text{th}$ iteration, $\alpha_i= \frac{\alpha}{(i+1)\log^2(i +1)}\beta$ \cite{tian2021online}, where $\beta$ is chosen so that $\sum_{i}\alpha_{i}=\alpha$.
 This schedule is beneficial in the case of an IterativeCUR because we often expect to make stopping decisions at later iterations rather than earlier ones. This weighting schedule ensures that enough weight is allocated to the later iterations that it is still reasonably likely to stop at the correct iteration.} For this bound to be well defined, we need ${w}>-4\log(\alpha)$. This means that for every desired $\alpha$ there is a {minimum number of extra columns, ${w}$,} that can be chosen to ensure the desired stopping control. For example, we can ensure an incorrect stopping probability of $10^{-10}$ by setting $w = 100$ and dividing $\epsilon$ by 4.98. In practice, using this adjustment at ${w} = 100$ is unnecessary (see \cref{sec:exp}). 
    
\subsection{Computational Complexity and Memory Requirements}
We can now consider the computational complexity of IterativeCUR. In \cref{tab:comp-complex} we display a complete breakdown of the computational cost for the dense matrix case when LUPP is used for row and column selection. For these costs, $b$ is the block size and $c$ is the sketch size, which in our implementation is set to $c = 1.1b$. 

The primary costs of IterativeCUR are: (1) computing the column residual (Line \ref{alg:left-res} of \icur), (2) computing the row residual (Line \ref{alg:right-res} of \icur), (3) updating the pseudoinverse (Line \ref{alg:update} of \icur), and (4) computing the initial sketch (Line \ref{alg:initial-sketch} of \icur). 

The primary expense in the residual computation is the matrix-matrix multiplications, multiplying the $\mat G \mat C \mat U$ with $\mat R$ in the column residual case and multiplying $\mat C$ with $\mat U \mat R[:,J_k]$ in the row residual case. This cost can be substantially reduced when the $\mat C$ and $\mat R$ factors are sparse, which typically occurs when $\mat A$ is sparse. Furthermore, because matrix-matrix multiplication is a BLAS-3 operation, parallelization allows this operation to be executed in a substantially more scalable manner than the $\mathcal{O}((m + n)r^2)$ cost suggests.

When the residual computation is not the dominant cost, computing the pseudoinverse for the core matrix becomes the dominant cost. If this pseudoinverse is recomputed at every iteration, this cost is $\mathcal{O}(r^4)$. However, we can reduce this cost by recognizing that each update to the CUR appends blocks of rows and columns to the intersection matrix, which allows us to use Givens rotations-based updating of the pseudoinverse described in \cite{stewart1998matrix}. Such pseudoinverse updates only require a complexity of $\mathcal{O}((bk)^2)$ at each iteration $k$. In practice, we rarely see benefits from this updating procedure over the recomputation of the pseudoinverse because of LAPACK's highly efficient parallelization of the QR factorization \cite{buttari2008parallel}. 
The update to QR can still be beneficial when small block sizes are used or when the rank of the approximation is very high, e.g. $r > 10,000$ depending on the architecture of the system. When $\epsilon$ is not too small,
one could also improve the practicality of the pseudoinverse phase by computing the LU decomposition rather than the QR. The LU is typically more parallelizable and equally as accurate provided that the matrix is not too poorly conditioned, as would be the case when $\epsilon$ is not too small and the matrix is not exactly low-rank. Additionally, if one wanted to obtain an approximation with orthonormal and diagonal factors, one can compute two QR factorizations $\mat{C}=\mat{Q_CR_C}$, $\mat{R}=\mat{R_R^\top Q^\top_R}$, compute the SVD $\mat{R}_{\mat{C}}\mat{UR}_{\mat{R}}^T=\mat{U_U\Sigma_UV_U}^T$, and output $\mat{A}\approx (\mat{Q_CU_U})\mat{\Sigma_U} (\mat{Q_R V_U})^T$. The additional cost is $O((m+n)r^2)$, so the complexity of IterativeCUR remains the same. 

The $1.1bn(2m-1)$ cost of computing the sketch of the matrix $\mat A$ could potentially be problematic for very large $m$ and $n$. This problem can be reduced if we sketch using a sparse-sign sketch \cite{cohen2016nearly} or a Subsampled Randomized Trigonometric Transform sketch \cite{tropp2019streaming}. In our experiments, we used a Gaussian matrix because for small sketch sizes $b$, all sketching techniques have roughly the same complexity.

In terms of storage, in the dense case, we need $mr$
storage for $\mat C_k$,  $r^2$ storage for $\mat U_k$, and $rn$ storage for $\mat R_k$. Additionally, we need $O(bn)$ space to store the sketch of $\mat A$ and the column residual. Finally, we need $O(bm)$ space to store the row residual. The small space required for the residuals allows for reduced data-movement costs, which lead to highly efficient pivoting on large matrices. When $\mat A$ is sparse, the storage for $\mat C_k,\mat R_k$ will be reduced accordingly.

	\begin{center}
		\begin{table}
            \label{tab:comp-complex}
\caption{Computational Complexity of IterativeCUR (\icur)}                     
		    \begin{tabular}{|c|c|}
                \hline
		        Action &  Total Cost\\
		        \hline
                Sketch & $1.1bn(2m-1)$\\
                LUPP for Cols & $rb (n - 1/2) + 7/6 r - rb^2 / 3$\\
                LUPP for Rows & $rb (m - 1/2) + 7/6 r - rb^2 / 3$\\
                Updated QR for U & $1/12 (r^3 (14 - 8/b) + r^2 (36b - 6/b) + r(43b^2 - 7b  -12))$\\
		        Col Residual & $c/(6b) (r^3 + r^2(6n - b + 3) + r(b^2 - 3) + bn$ \\
		        Row Residual & $1/6 (r^3 + r^2 (6m - b + 3) + r(b^2 - 3) + bm$ \\
		       {Norm of Col Residual} & {$2bn - 1$}\\
                \hline
                
	        \end{tabular}            
	    \end{table}
        
	\end{center}

\section{Understanding the Performance of IterativeCUR}\label{sec:the}
With an understanding of the algorithm, we now focus on the theoretical performance of IterativeCUR (\icur). 
In this section, 
we first describe the connections between IterativeCUR and existing algorithms in the literature to indicate why it works well in practice.  
We then derive a bound that provides some insight into the inner workings of IterativeCUR. 

\subsection{{Connections to existing algorithms}}\label{subsec:connection}
Adaptive pivoting strategies form a powerful class of methods for the column subset selection problem (CSSP). 
Possibly the first such algorithm was introduced by 
Deshpande and Vempala~\cite{deshpande2006adaptive}, where a column is chosen with probability proportional to the squared column norm of the current residual with an interpolative decomposition, i.e., $\mat A-\mat{CC}^\dagger \mat A$. 
When $\mat A$ is positive semidefinite, significant simplifications ensue, resulting in the Randomly pivoted Cholesky algorithm, for which an efficient implementation and extensive theory have been developed in~\cite{chen2025randomly}. 

Let us now consider IterativeCUR from this viewpoint. Just like the adaptive sampling method of Deshpande-Vempala's, we update the sampling strategy based on the current residual. There are two notable differences: (i) Our goal is CUR, not $\mat{CC}^\dagger \mat{A}$ (which is significantly more expensive); accordingly our residual is  $\mat A-\mat{CUR}$ instead of $\mat A-\mat C\mat C^\dagger \mat A$. 
(ii) In IterativeCUR we find the pivots via (usually LUPP applied to) a sketched residual $\mat G(\mat{A}-\mat{CUR})$. As discussed above, this is a powerful strategy to find pivots for CSSP applied to the residual, and more efficient than computing the column norms of the residuals every time. 
Despite these differences, IterativeCUR can be seen as a variant of Deshpande-Vempala's method where CUR is the objective and a sketched pivoting strategy is used (which is empirically at least equally powerful) instead of norm-wise sampling. These close connections to existing adaptive sampling methods explain why IterativeCUR is expected to work, at least qualitatively.

\subsection{Error bound}\label{subsec:error}
With intuition for the expected behavior of IterativeCUR, we now derive an error bound for IterativeCUR. Deriving this bound relies on an index selection method proposed by \cite{osinsky2025close} as opposed to the QR with column pivoting (QRCP) and LU with partial pivoting (LUPP) approaches suggested in \cref{sec:alg}. We do so because \cite{osinsky2025close} has stronger and more interpretable theoretical bounds than the LUPP or QRCP approaches, and in practice performs only slightly better than the LUPP or QRCP selection approaches \cite{cortinovis2024adaptive}. 

To derive our bound, we rely on the cross-approximation result from Osinsky \cite{osinsky2025close}. This result applies to the scenario where the row and column indices are selected using \cite[Algorithm 1]{osinsky2025close} and a cross-approximation is used to compute the matrix $\mat U$, as we do in IterativeCUR (\icur). In this scenario, \cite{osinsky2025close} obtains
\begin{theorem}[Theorem 2 \cite{osinsky2025close}] \label{thm:osinsky}
    Let $\mat A,\mat Z\in \real^{m \times n}$ where $\mat Z$ is a rank-$r$ approximation to $\mat A$. Then it is possible to find columns indices, $J$, and row indices, $I$ of the matrix $\mat A$, such that 
    \begin{equation}
        \|\mat A - \mat A(:,J) \mat A(I,J)^\dagger \mat A(I,:)\|_F \leq (r+1) \|\mat A - \mat Z\|_F
    \end{equation}
    and that 
    \begin{equation}
        \|\mat A - \mat A(:,J) \mat A(I,J)^\dagger \mat A(I,:)\|_2 \leq \sqrt{1 + r(r+2)(\min(m,n) - r)} \|\mat A - \mat Z\|_2.
    \end{equation}
\end{theorem}

It is important to note that this bound applies for any rank-$r$ subspace. To make use of this result, we first derive a recursive relationship on the CUR residual. Then we repeatedly apply \cref{thm:osinsky} to rank-$b$ subspace approximations to the residual at each iteration. For the residual recursion of CUR residual we have

\begin{lemma} \label{lemma:recurrence}
Let $\mat A \in \real^{m \times n}$, and let $\mat A_{I,J} = \mat A(:, J) \mat A(I,J)^{{-1}} \mat A(I,:)$ be the CUR approximation to the matrix $\mat A$ with row indices, $I$, and columns indices $J$ {where $|I| = |J|$ and $\mat{A}(I,J)$ is invertible}. {Define the residual $\mat{S} = \mat{A} - \mat{A}_{I,J}$.} Then for any other set of row indices $I_+$ and $J_+$ {satisfying $I \cap I_+ = \emptyset$, $J \cap J_+ = \emptyset$, $|I_+| = |J_+|$ and such that $\mat S(I_+,J_+)$ is invertible,} we have 
    \begin{equation}
        \mat{A} - \mat{A}_{I\cup I_+, J\cup J_+} = \mat{A} - \mat{A}_{I,J} - \mat{S}_{I_+,J_+} ( = \mat{S} - \mat{S}_{I_+,J_+})
    \end{equation} where 
    $\mat{S}_{I_+,J_+} = \mat S(:, J_+) \mat S(I_+,J_+)^{{-1}} \mat S(I_+,:)$.
\end{lemma}
\begin{proof}
First, by expanding $\mat{A}_{I\cup I_+, J\cup J_+}$ we get
    \begin{align*}
        \mat{A}_{I\cup I_+,J\cup J_+} = \onetwo{\mat{A}(:,J)}{\mat{A}(:,J_+)}\twotwo{\mat{A}(I,J)}{\mat{A}(I,J_+)}{\mat{A}(I_+,J)}{\mat{A}(I_+,J_+)}^{{-1}} \twoone{\mat{A}(I,:)}{\mat{A}(I_+,:)}.
    \end{align*} Now define $\mat{C} = \mat{A}(:,J),\mat{C}_0 = \mat{A}(:,J_+), \mat{R} = \mat{A}(I,:),\mat{R}_0 = \mat{A}(I_+,:), \mat{U}_{11} = \mat{A}(I,J),$ $\mat{U}_{10} = \mat{A}(I,J_+),\mat{U}_{01} = \mat{A}(I_+,J),\mat{U}_{00} = \mat{A}(I_+,J_+)$ for shorthand. Let us use the block inversion formula to expand the pseudoinverse term.
    \begin{equation}
        \twotwo{\mat{U}_{11}}{\mat{U}_{10}}{\mat{U}_{01}}{\mat{U}_{00}}^{{-1}} = \twotwo{\mat{U}_{11}^{{-1}} + \mat{U}_{11}^{{-1}} \mat{U}_{10} \mat{N}^{-1}\mat{U}_{01}\mat{U}_{11}^{{-1}} }{-\mat{U}_{11}^{{-1}} \mat{U}_{10} \mat{N}^{{-1}}}{-\mat{N}^{{-1}} \mat{U}_{01}\mat{U}_{11}^{{-1}} }{\mat{N}^{{-1}}}
    \end{equation} where $\mat{N} = \mat{A}(I_+,J_+) - \mat{A}(I_+,J)\mat{A}(I,J)^{{-1}} \mat{A}(I,J_+) = \mat{U}_{00} - \mat{U}_{01}\mat{U}_{11}^{{-1}} \mat{U}_{10}$. Therefore,
    \begin{align*}
        \mat{A}_{I\cup I_+,J\cup J_+}  &= \mat{C}\mat{U}_{11}^{{-1}} \mat{R} + \mat{C}\mat{U}_{11}^{{-1}} \mat{U}_{10}\mat{N}^{{-1}} \mat{U}_{01}\mat{U}_{11}^{{-1}} \mat{R} - \mat{C}_0 \mat{N}^{{-1}} \mat{U}_{01}\mat{U}_{11}^{{-1}} \mat{R}\\
        & \indent  - \mat{C}\mat{U}_{11}^{{-1}} \mat{U}_{10}\mat{N}^{{-1}} \mat{R}_0 + \mat{C}_0 \mat{N}^{{-1}} \mat{R}_0 \\
        &= \mat{C}\mat{U}_{11}^{{-1}} \mat{R} + (\mat{C}_0 - \mat{C}\mat{U}_{11}^{{-1}} \mat{U}_{10})\mat{N}^{{-1}} (\mat{R}_0 - \mat{U}_{01}\mat{U}_{11}^{{-1}} \mat{R}) \\
        &= \mat{A}_{I,J} + \mat{S}(:,J_+) \mat{S}(I_+,J_+)^{{-1}} \mat{S}(I_+,:) \\
        &= \mat{A}_{I,J} + \mat{S}_{I_+,J_+},
    \end{align*} which concludes the proof.
\end{proof}

{
\begin{remark}
    The invertibility assumption in \Cref{lemma:recurrence} is very mild, assuming a sensible pivoting strategy is used; 
    if there is no suitable choice of $I_+$ and $J_+$ 
    such that $\mat{S}(I_+,J_+)$ is numerically nonsingular,   
    this means that the residual has numerical rank smaller than the requested block size, in which case a smaller final block may be used.
\end{remark}
}

\cref{lemma:recurrence} can naturally be thought of as a variant of a continuation of a block LU with complete pivoting procedure with the starting pivots $I$ and $J$ and to which a set of new pivots $I_+$ and $J_+$ is appended. Also, the CUR decomposition of $\mat{A}$ with $I\cup I_+$ and $J\cup J_+$ equals the CUR decomposition of $\mat{A}$ with $I$ and $J$ plus the CUR decomposition of the residual matrix $\mat{S} = \mat{A}-\mat{A}_{I,J}$ with $I_+$ and $J_+$. For a series of sets of indices $I_1,..., I_k$ and $J_1,..., J_k$, \cref{lemma:recurrence} implies
\begin{equation} \label{eq:recurrencenorm}
    \norm{\mat{A} - \mat{A}_{\bigcup\limits_{j=1}^k I_j,\bigcup\limits_{j=1}^k J_j}} = \norm{\mat{A}}\prod_{i = 1}^k \frac{\norm{\mat{S}^{(i)} - \mat{S}_{I_i,J_i}^{(i)}}}{\norm{\mat{S}^{(i)}}}
\end{equation} where $\mat{S}^{(i)} = \mat{A} - \mat{A}_{\bigcup\limits_{j=1}^{i-1} I_j,\bigcup\limits_{j=1}^{i-1} J_j}$ and $\mat{S}^{(1)} = \mat{A}$.

As \cref{alg:ICUR} obtains a set of $b$ row and column indices per iteration, it allows us to establish the following result. Here, $b$ is the block size.

\begin{theorem} \label{thm:converge}
    Suppose we run $k$ iterations of \cref{alg:ICUR} with block size $b$ where we use Algorithm 1 in \cite{osinsky2025close} on the row sketch of the residual matrix $\mat{S}^{(i)}$ to obtain the row and column indices. Then we obtain a rank-$(bk)$ CUR approximation to $\mat{A}$ and the CUR decomposition satisfies
    \begin{equation} \label{eq:errbnd_attemp1}
        \norm{\mat{A} - \mat{A}_{\bigcup\limits_{j=1}^k I_j,\bigcup\limits_{j=1}^k J_j}}_F \leq (1+b)^k \norm{\mat{A}}_F \prod\limits_{i = 1}^k \frac{\norm{\mat{S}^{(i)}-\mat{S}^{(i)}\mat{X}_i^\dagger \mat{X}_i}_F}{\norm{\mat{S}^{(i)}}_F}
    \end{equation} where $I_j$ and $J_j$ are the row and column indices obtained in the $j$th iteration of \cref{alg:ICUR}, and
    \begin{equation}
        \mat{X}_i = \mat{G}\mat{S}^{(i)} = \mat{G}\left(\mat{A} - \mat{A}_{\bigcup\limits_{j=1}^{i-1} I_j,\bigcup\limits_{j=1}^{i-1} J_j}\right).
    \end{equation} The matrix $\mat{G}\in \real^{b\times m}$ is a sketching matrix.
\end{theorem}
\begin{proof}
    The proof follows by applying result (14) in Theorem 2 of \cite{osinsky2025close} on \cref{eq:recurrencenorm}.
\end{proof}
\begin{remark}
    As $\mat{S}^{(i)}$ is the residual error matrix of the CUR decomposition with the indices $\bigcup_{j=1}^{i-1} I_j$ and $\bigcup_{j=1}^{i-1} J_j$ and hence has connections to the Schur complement, it is generally difficult to quantify the singular values of $\mat{S}^{(i)}$. Roughly, if the set of indices $\bigcup_{j=1}^{i-1} I_j$ and $\bigcup_{j=1}^{i-1} J_j$ are good we expect
\begin{equation}
    \norm{\mat{S}^{(i)}}_F \lesssim \sqrt{\sum_{\ell = b(i-1)+1}^n \sigma_\ell (\mat{A})^2}.
\end{equation} On the other hand, if we ignore the dependency of the sketching matrix $\mat{G}$ with the sets of indices $I_1,..,I_k$ and $J_1,...,J_k$, then $\norm{\mat{S}^{(i)}-\mat{S}^{(i)}\mat{X}_i^\dagger \mat{X}_i}_F$ is the randomized SVD error of $\mat{S}^{(i)}$ with target rank $b$. Therefore, we roughly have
\begin{equation}
    \norm{\mat{S}^{(i)}-\mat{S}^{(i)}\mat{X}_i^\dagger \mat{X}_i}_F \lesssim \sqrt{\sum_{\ell = b+1}^n \sigma_\ell \left(\mat{S}^{(i)} \right)^2} \lesssim \sqrt{\sum_{\ell = bi+1}^n \sigma_\ell (\mat{A})^2}.
\end{equation} If we cancel out the terms in the product \eqref{eq:errbnd_attemp1} in a telescoping way, we very roughly obtain
\begin{equation}
    \norm{\mat{A} - \mat{A}_{I,J}}_F \lesssim (1+b)^k \sqrt{\sum\limits_{\ell = bk+1}^n \sigma_\ell (\mat{A})^2} = (1+b)^k \norm{\mat{A} - \lowrank{\mat{A}}{bk}}_F.
\end{equation}
\end{remark}

In experiments, we have not observed any cases where IterativeCUR has a relative error worse than twice that of the error from the CUR produced using \cite{osinsky2025close}, even when testing Matlab's gallery matrices with slow spectral decay. This indicates that the suboptimality factor in \cref{thm:converge} is far more conservative than what is observed in practice.  Despite this conservativeness, \cref{thm:converge} provides intuition for IterativeCUR's worst-case performance if IterativeCUR behaved as a wholly greedy algorithm. Specifically, this analysis reveals that at every iteration, the index selection phase aims to capture the best rank-$b$ approximation that, when using \cite{osinsky2025close}, will be a $(b+1)$-suboptimal approximation in the worst case. What is not well-accounted for in this analysis is how using the residual facilitates the selection of indices that correct for suboptimal index selections from previous iterations. This correction seems to be why, as will be demonstrated in the next section, the net effect of these costs and benefits seems to be IterativeCUR performing the same as the sketched and pivoted techniques.

\section{Numerical Experiments}\label{sec:exp}
With an understanding of the details of IterativeCUR, we now examine its practical effectiveness through two broad types of experiments. Both experiments were performed on a 48-core, 2.6/4.0 GHz  Intel(R) Xeon(R) Gold 6442Y CPU with 2 TB of memory. All implementations were written and run in MATLAB version 2025a\footnote{Code to recreate all experiments can be found at \url{https://github.com/nathanielpritchard/IterativeCUR_code}.}. In the first type of experiment, we perform 10 runs of IterativeCUR\footnote{Although in \cref{alg:ICUR} we use $\|\mat G \mat A\|_F$ in the denominator, because the experiments use a Gaussian sketching matrix and we need to compute $\|\mat A\|_F$ for the final error computation of the other methods, we considered using the relative error with $\|\mat A\|_F$ in the denominator.} with a fixed stopping threshold and collect information about its accuracy, speed, and rank of approximation. In the second type of experiment, we perform five runs of IterativeCUR for each rank in a sequence of ranks and examine its median accuracy and median computation time. The first set of experiments gives us an indication of how IterativeCUR performs when a solution of a specified tolerance is desired. The second set of experiments indicates how effectively IterativeCUR collects indices.

We compare the performance of IterativeCUR using LUPP column and row index selection with two methods, rank-adaptive Randomized SVD (SVDSketch), proposed in \cite{yu2018efficient} \footnote{ Can be called using \texttt{svdsketch} in MATLAB.}, and sketched LU with Partial Pivoting (s-LUPP), proposed in \cite{dong2023simpler}. By virtue of being an approximation of the SVD of $\mat A$, SVDSketch provides a good baseline for the best approximation performance at a particular rank that a rank-adaptive technique can achieve. S-LUPP, which is a very high-quality CUR method in terms of computational time and accuracy (but is not rank-adaptive), provides us with an effective baseline for how accurate an effectively chosen CUR can be on a particular problem.  This comparison is particularly important because as CUR is constrained to selecting subsets of rows and columns, the best possible error of a CUR approximation at a particular rank, $r$, can be up to a factor of $r+1$ worse than the truncated SVD \cite{osinsky2025close}. 

For our experiments, we first consider how well IterativeCUR performs in its intended use, returning an approximation of a desired quality. Here we use the first experiment type to compare IterativeCUR, s-LUPP, and SVDSketch to each other in terms of accuracy, computational time, and rank when run to a specific tolerance. Then, with an understanding of how well IterativeCUR performs in its designed context, we next investigate how efficiently IterativeCUR selects columns by comparing the accuracy of each method's approximation at fixed ranks. We close this section with investigations into the effects that index selection method and block size have on approximation accuracy of IterativeCUR.

\subsection{Threshold Examples} \label{subsec:threshold}
We first examine the performance of IterativeCUR in the rank-adaptive regime. Here we consider approximating the two square matrices displayed in \cref{tab:rank-adapt}. The first matrix is the \textsc{Low-Rank} matrix, which is exactly low-rank and is formed by generating two standard Gaussian matrices, $\mat{G}_i \in \real^{30000 \times 2000}, i \in \{1,2\}$, and setting $\mat A = \mat{G}_1 \mat{G}_2^\top$. The second matrix is the \textsc{c-67} matrix from the SuiteSparse library \cite{davis2011university}. 
\begin{table}[h]
    \centering
    \caption{Matrix information for fixed rank experiment.} \label{tab:rank-adapt}
    \begin{tabular}{c|c|c|c}
    \toprule
    Matrix & Dimension & Sparse & Tolerance \\
   \hline
    \textsc{Low-Rank} & 30,000 & False & $1\times 10^{-6}$ \\
     \textsc{c-67} & 57,975 & True & $1\times10^{-2}$ \\
    \end{tabular}
  \end{table}

In the experiment, we run each method (IterativeCUR, s-LUPP, SVDSketch), 10 times. SVDSketch and IterativeCUR were run with a block size of 250 and stopped when the relative error reached the threshold specified in \cref{tab:rank-adapt}. The results of these experiments are displayed in \cref{fig:c-67,fig:low_rank_comp} as box-and-whisker plots where the middle line represents the median, the outer lines on the box represent the 25th and 75th percentiles, and the outer endpoints represent the minimum and maximum. 

\begin{figure}[h]
    \centering
    \begin{tikzpicture}[scale = \boxplotscale, anchor = west]
	\pgfplotstableread[col sep=comma]{./csvs/synthetic/Low_Rank_30000_250_0_overall_accuracy.csv}\accdata
    \pgfplotstableread[col sep=comma]{./csvs/synthetic/Low_Rank_30000_250_0_n_cols.csv}\coldata
    \pgfplotstableread[col sep=comma]{./csvs/synthetic/Low_Rank_30000_250_0_time.csv}\timedata
    \newcommand{\mattype}{\textsc{Low-Rank} }
    
    \pgfplotscreateplotcyclelist{cross_colors}{cLU, cILU, cSVDs}
	\begin{axis}[
         width = .4 \textwidth,
        height = .4 \textwidth,
        name = axis1,
		boxplot/draw direction = x,
		x axis line style = {opacity=0},
		axis x line* = bottom,
		axis y line = left,
		enlarge y limits,
		xmajorgrids,
        title = {Accuracy for \mattype},
		ytick = {1, 2, 3, 4},
		yticklabel style = {align=center, font=\small, rotate=0},
		yticklabels = {\curs, \icurl, \svds},
		ytick style = {draw=none}, 
		xlabel = {$\frac{\|A - \hat{A}_k\|_F}{\|A\|_F}$},
        cycle list name=cross_colors,
	]
		\foreach \n in {1,...,3} {
			\addplot+[boxplot, fill, draw, draw opacity = \dopa, fill opacity = \fopa] table[y index=\n] {\accdata};
		}
	\end{axis}
        	\begin{axis}[
        width = .4 \textwidth,
        height = .4 \textwidth,
        name = axis2,
        at={(axis1.outer north east)},
        anchor=outer north west,
        xshift = 1 cm,
		boxplot/draw direction = x,
		x axis line style = {opacity=0},
		axis x line* = bottom,
		axis y line = left,
		enlarge y limits,
		xmajorgrids,
        title = {Rank for \mattype},
		ytick = {1, 2, 3, 4},
		yticklabel style = {align=center, font=\small, rotate=0},
        yticklabel = \empty,
		ytick style = {draw=none}, 
		xlabel = {Rank},
        cycle list name=cross_colors,
	]
		\foreach \n in {1,...,3} {
			\addplot+[boxplot, fill, draw, draw opacity = \dopa, fill opacity = \fopa] table[y index=\n] {\coldata};
		}
	\end{axis}
            	\begin{axis}[
		 width = .4 \textwidth,
        height = .4 \textwidth,
		boxplot/draw direction = x,
         at={(axis2.outer north east)},
         xshift = 1 cm,
        anchor=outer north west,
		x axis line style = {opacity=0},
		axis x line* = bottom,
		axis y line = left,
		enlarge y limits,
		xmajorgrids,
        title = {Time for \mattype},
		ytick = {1, 2, 3, 4},
		yticklabel style = {align=center, font=\small, rotate=0},
        yticklabel = \empty,
		ytick style = {draw=none}, 
		xlabel = {Time (seconds)},
        cycle list name=cross_colors
	]
		\foreach \n in {1,...,3} {
			\addplot+[boxplot, fill, draw, draw opacity = \dopa, fill opacity = \fopa] table[y index=\n] {\timedata};
		}
	\end{axis}
\end{tikzpicture}
    \caption{Box plots for the performance of SVDSketch, s-LUPP, and IterativeCUR-LUPP applied to a \textsc{low-rank} matrix in \cref{tab:rank-adapt}. The left plot displays the accuracy of the approximation, the middle plot displays the rank of the approximation, and the right plot displays the computational time.}
    \label{fig:low_rank_comp}
\end{figure}

\begin{figure}[h]
    \centering
    \begin{tikzpicture}[scale = \boxplotscale, anchor = west]
	\pgfplotstableread[col sep=comma]{./csvs/synthetic/c-67_57975_250_0_overall_accuracy.csv}\accdata
    \pgfplotstableread[col sep=comma]{./csvs/synthetic/c-67_57975_250_0_n_cols.csv}\coldata
    \pgfplotstableread[col sep=comma]{./csvs/synthetic/c-67_57975_250_0_time.csv}\timedata
    \newcommand{\mattype}{\textsc{c-67} }
    
    \pgfplotscreateplotcyclelist{cross_colors}{cLU, cILU, cSVDs}
	\begin{axis}[
        width = .4 \textwidth,
        height = .4 \textwidth,
        name = axis1,
		boxplot/draw direction = x,
		x axis line style = {opacity=0},
		axis x line* = bottom,
		axis y line = left,
		enlarge y limits,
		xmajorgrids,
        title = {Accuracy for \mattype},
		ytick = {1, 2, 3, 4},
		yticklabel style = {align=center, font=\small, rotate=0},
		yticklabels = {\curs, \icurl, \svds},
		ytick style = {draw=none}, 
		xlabel = {$\frac{\|A - \hat{A}_k\|_F}{\|A\|_F}$},
        cycle list name=cross_colors,
	]
		\foreach \n in {1,...,3} {
			\addplot+[boxplot, fill, draw, draw opacity = \dopa, fill opacity = \fopa] table[y index=\n] {\accdata};
		}
	\end{axis}
        	\begin{axis}[
        width = .4 \textwidth,
        height = .4 \textwidth,
        name = axis2,
        at={(axis1.outer north east)},
        anchor=outer north west,
        xshift = 1 cm,
		boxplot/draw direction = x,
		x axis line style = {opacity=0},
		axis x line* = bottom,
		axis y line = left,
		enlarge y limits,
		xmajorgrids,
        title = {Rank for \mattype},
		ytick = {1, 2, 3, 4},
		yticklabel style = {align=center, font=\small, rotate=0},
        yticklabel = \empty,
		ytick style = {draw=none}, 
		xlabel = {Rank},
        cycle list name=cross_colors,
	]
		\foreach \n in {1,...,3} {
			\addplot+[boxplot, fill, draw, draw opacity = \dopa, fill opacity = \fopa] table[y index=\n] {\coldata};
		}
	\end{axis}
            	\begin{axis}[
        width = .4 \textwidth,
        height = .4 \textwidth,
		boxplot/draw direction = x,
         at={(axis2.outer north east)},
         xshift = 1 cm,
        anchor=outer north west,
		x axis line style = {opacity=0},
		axis x line* = bottom,
		axis y line = left,
		enlarge y limits,
		xmajorgrids,
        title = {Time for \mattype},
		ytick = {1, 2, 3, 4},
		yticklabel style = {align=center, font=\small, rotate=0},
        yticklabel = \empty,
		ytick style = {draw=none}, 
		xlabel = {Time (seconds)},
        cycle list name=cross_colors
	]
		\foreach \n in {1,...,3} {
			\addplot+[boxplot, fill, draw, draw opacity = \dopa, fill opacity = \fopa] table[y index=\n] {\timedata};
		}
	\end{axis}
\end{tikzpicture}
    \caption{Box plots for the performance of SVDSketch, s-LUPP, and IterativeCUR-LUPP applied to a \textsc{c-67} matrix in \cref{tab:rank-adapt}. The left plot displays the accuracy of the approximation, the middle plot displays the rank of the approximation, and the right plot displays the computational time.}
    \label{fig:c-67}
\end{figure}
In the experiments for the \textsc{Low-Rank} matrix we can observe that all three methods return an approximation with the original matrix's rank, 2000, and are far more accurate than the desired threshold, $1\times 10^{-6}$, with the median accuracy for SVDSketch being an order of magnitude more accurate at $4 \times 10^{-15}$ than the CUR approaches (both $9 \times 10^{-14}$). The major difference between the methods is observed in the timings where IterativeCUR is the fastest method with a median time of 12 seconds, slightly faster than s-LUPP's median time of 15 seconds, and about 7 times faster than SVDSketch's median time of 85 seconds. Overall, from this experiment, we can see in the easy-to-handle exactly low-rank case, IterativeCUR performs effectively. 

When we make the experiment more difficult by considering a large sparse \textsc{c-67} matrix, we still see that IterativeCUR performs well. Here, the stopping threshold is $1\times 10^{-2}$, and we observe that IterativeCUR finds solutions that are slightly more accurate than SVDSketch and similarly accurate to s-LUPP with the added advantage of being less variable than s-LUPP. Further, IterativeCUR uses only a median time of 2.6 seconds to find this solution as opposed to the median time of 3.9 seconds for s-LUPP, and the median time of 41 seconds for SVDSketch. Overall, we can see that IterativeCUR can perform highly effectively in the rank-adaptive context.

\subsection{Efficiency of Selection}\label{subsec:fixed_rank}
While it is clear that IterativeCUR can effectively form approximations of a single specified quality, it is unclear whether IterativeCUR can efficiently form such approximations for an arbitrary rank. To examine the efficiency of the approximation, we compare the speed and accuracy of IterativeCUR's approximation with the s-LUPP and SVDSketch approximations of fixed ranks $\{500,1000,1500,2000,2500, 3000, 3500, 4000\}$ for the matrices listed in \cref{tab:fixed-rank}. The \textsc{Low-Rank PD} matrix is generated the same way as the \textsc{low-rank} matrix in \cref{subsec:threshold} with the addition of a diagonal matrix containing on its diagonal the vector $(1-\exp(-{\log(30,000)}/{\log(5)})^i \text{, } i \in \{1, 2, \dots, 30000\}$. The \textsc{G7Jac200} and \textsc{Bayer01} matrices come from SuiteSparse \cite{davis2011university}.  We run each method, matrix, and rank combination five times and display the median runtime, accuracy, total number of Floating Point Operations \footnote{For sparse matrices, we approximate sparse times dense matrix multiplications used for sketching with the number of floating point operations for all matrices.}, and Floating Point Operations per second in \cref{fig:fixed_rank_low_rank_pd},  \cref{fig:fixed_rank_g7jac}, and \cref{fig:fixed_rank_bayer01}. For each run, IterativeCUR and SVDSketch are run with a block size of 250.
\begin{table}[h]
    \centering
    \caption{Matrix information for fixed-rank experiment.} \label{tab:fixed-rank}
    \begin{tabular}{c|c|c}
    \toprule
    Matrix & Dimension & Sparse \\
   \hline
    \textsc{Low-Rank PD} & 30,000 & False\\
     \textsc{G7Jac200} & 57,310 & True\\
     \textsc{Bayer01} & 57,735 &True\\
    \end{tabular}
  \end{table}

From these results, we can see that the median relative errors of the approximations formed by IterativeCUR are as small, if not smaller, than the median s-LUPP relative error at every rank. Additionally, the IterativeCUR relative error typically is only at most twice the relative error of SVDSketch.  

In terms of runtime, IterativeCUR is much faster than SVDSketch. In the best case, IterativeCUR is 36 times faster than SVDSketch when estimating a rank-4000 approximation for the \textsc{G7Jac200} matrix, and in the worst case, IterativeCUR is still six times faster than SVDSketch for the rank 500 approximations to the \textsc{G7Jac200} and \textsc{Low-Rank PD} matrices.  This time difference is much smaller for s-LUPP with IterativeCUR being only four times faster in the best case, rank 4000 approximation to \textsc{Low-Rank PD}, and providing no speed-up in the worst cases, the rank 500 approximations for the \textsc{G7Jac200} and \textsc{Bayer01} matrices.  Compared to both alternatives, it is also important to note that the benefits are greatest for the densest matrix, \textsc{Low-Rank PD}, where the memory load is the highest. This seems to suggest that on much larger matrices with larger memory loads, the benefits afforded by a single small sketch for IterativeCUR should lead to even greater performance improvements over these alternatives.  

We also can find more evidence for this hypothesis when we factor in median FLOP counts and median FLOP rates. For FLOP counts, we can see that IterativeCUR offers around a five to ten time reduction in FLOP count for all matrices compared to s-LUPP and SVDSketch. However, IterativeCUR also has FLOP rates that are about a factor four worse than s-LUPP in the worst case (rank 4000 approximation of Bayer01 matrix). As s-LUPP has almost all of its cost come from the matrix multiplication required for sketching the matrix, it is not surprising that its FLOP rates are better than IterativeCUR. Similarly, because SVDSketch requires orthogonalization, it is also not surprising that it has the slowest FLOP rate compared to the other two methods. 

\begin{figure}[h]
    \centering
    \begin{tikzpicture}[scale = \boxplotscale]
    \newcommand{\matname}{Low_Rank_PD}
    \newcommand{\Matname}{\textsc{Low-Rank PD}}
    \pgfplotstableread[col sep=comma]{./csvs/same_rank/\matname_250_accuracy.csv}\accdata
    \pgfplotstableread[col sep=comma]{./csvs/same_rank/\matname_250_time.csv}\timedata
     \pgfplotstableread[col sep=comma]{./csvs/same_rank/\matname_250_flops.csv}\flopdata
 \begin{axis}[
        name = axis1,
        ylabel = {Relative Frobenius Error},
        xlabel = {Rank of Approximation},
        title = {Accuracy for \Matname},
        ymode = log,
        ymajorgrids,
        legend pos = north east,
        scale only axis,
        xlabel near ticks,
        ylabel near ticks,
        ]
        \addplot[svds_line]table[x = n_cols, y = svds]\accdata;
        \addplot[scur_line]table[x = n_cols, y = LUPP_sketch_cur]\accdata;
        \addplot[icur_line]table[x = n_cols, y = LUPP_iterative_cur]\accdata;
        \legend{\svds, \curs, \icurl};
    \end{axis}
    
    \begin{axis}[
        name = axis2,
        at={(axis1.outer north east)},
        anchor=outer north west,
        xshift = 0.5em,
        ylabel = {Time (Seconds)},
        xlabel = {Rank of Approximation},
        title = {Time for \Matname},
        ymode = log,
        ymajorgrids,
        legend pos = south east,
        scale only axis,
        xlabel near ticks,
        ylabel near ticks,
        ]
        \addplot[svds_line]table[x = n_cols, y = svds]\timedata;
        \addplot[scur_line]table[x = n_cols, y = LUPP_sketch_cur]\timedata;
        \addplot[icur_line]table[x = n_cols, y = LUPP_iterative_cur]\timedata;
    \end{axis}
    \begin{axis}[
        name = axis3,
        at={(axis1.outer south west)},
        anchor=outer north west,
        yshift = -0.5em,
        ylabel = {Floating Point Operations},
        xlabel = {Rank of Approximation},
        title = {Total FLOPs for \Matname},
        ymode = log,
        ymajorgrids,
        legend pos = north east,
        scale only axis,
        xlabel near ticks,
        ylabel near ticks,
        ]
        \addplot[svds_line]table[x = n_cols, y = flops_svds]\flopdata;
        \addplot[scur_line]table[x = n_cols, y = flops_LUPP_sketch_cur]\flopdata;
        \addplot[icur_line]table[x = n_cols, y = flops_LUPP_iterative_cur]\flopdata;
    \end{axis}
    \begin{axis}[
        name = axis4,
        at={(axis2.outer south west)},
        anchor=outer north west,
        yshift = -0.5em,
        ylabel = {FLOP Rate},
        xlabel = {Rank of Approximation},
        title = {FLOP Rate for \Matname},
        ymode = log,
        ymin = 1e10,
        ymax = 1e11,
        ymajorgrids,
        legend pos = north east,
        scale only axis,
        xlabel near ticks,
        ylabel near ticks,
        ]
        \addplot[svds_line]table[x = n_cols, y = flop_rate_svds]\flopdata;
        \addplot[scur_line]table[x = n_cols, y = flop_rate_LUPP_sketch_cur]\flopdata;
        \addplot[icur_line]table[x = n_cols, y = flop_rate_LUPP_iterative_cur]\flopdata;
    \end{axis}

\end{tikzpicture}
    \caption{This plot shows the change in median relative error, median runtime, median number of FLOPs, and median FLOP rate for \svds, \icurl, and \curs{} with approximations of varying rank for a \textsc{Low-Rank PD} matrix.}
    \label{fig:fixed_rank_low_rank_pd}
\end{figure}
\begin{figure}[h]
    \centering
    \begin{tikzpicture}[scale = \boxplotscale]
    \newcommand{\matname}{g7jac200}
    \newcommand{\Matname}{\textsc{G7Jac200}}
    \pgfplotstableread[col sep=comma]{./csvs/same_rank/\matname_250_accuracy.csv}\accdata
    \pgfplotstableread[col sep=comma]{./csvs/same_rank/\matname_250_time.csv}\timedata
     \pgfplotstableread[col sep=comma]{./csvs/same_rank/\matname_250_flops.csv}\flopdata
  \begin{axis}[
        name = axis1,
          ylabel = {Relative Frobenius Error},
        xlabel = {Rank of Approximation},
        title = {Accuracy for \Matname},
        ymode = log,
        ymajorgrids,
        legend pos = north east,
        scale only axis,
        xlabel near ticks,
        ylabel near ticks,
        ]
        \addplot[svds_line]table[x = n_cols, y = svds]\accdata;
        \addplot[scur_line]table[x = n_cols, y = LUPP_sketch_cur]\accdata;
        \addplot[icur_line]table[x = n_cols, y = LUPP_iterative_cur]\accdata;
        \legend{\svds, \curs, \icurl};
    \end{axis}
    
    \begin{axis}[
        name = axis2,
        at={(axis1.outer north east)},
        anchor=outer north west,
        xshift = 0.5em,
        ylabel = {Time (Seconds)},
        xlabel = {Rank of Approximation},
        title = {Time for \Matname},
        ymode = log,
        ymajorgrids,
        legend pos = south east,
        scale only axis,
        xlabel near ticks,
        ylabel near ticks,
        ]
        \addplot[svds_line]table[x = n_cols, y = svds]\timedata;
        \addplot[scur_line]table[x = n_cols, y = LUPP_sketch_cur]\timedata;
        \addplot[icur_line]table[x = n_cols, y = LUPP_iterative_cur]\timedata;
    \end{axis}
    \begin{axis}[
        name = axis3,
        at={(axis1.outer south west)},
        anchor=outer north west,
        yshift = -0.5em,
        ylabel = {Floating Point Operations},
        xlabel = {Rank of Approximation},
        title = {Total FLOPs for \Matname},
        ymode = log,
        ymajorgrids,
        legend pos = north east,
        scale only axis,
        xlabel near ticks,
        ylabel near ticks,
        ]
        \addplot[svds_line]table[x = n_cols, y = flops_svds]\flopdata;
        \addplot[scur_line]table[x = n_cols, y = flops_LUPP_sketch_cur]\flopdata;
        \addplot[icur_line]table[x = n_cols, y = flops_LUPP_iterative_cur]\flopdata;
    \end{axis}
    \begin{axis}[
        name = axis4,
        at={(axis2.outer south west)},
        anchor=outer north west,
        yshift = -0.5em,
        ylabel = {FLOP Rate},
        xlabel = {Rank of Approximation},
        title = {FLOP Rate for \Matname},
        ymode = log,
        ymajorgrids,
        legend pos = north east,
        scale only axis,
        xlabel near ticks,
        ylabel near ticks,
        ]
        \addplot[svds_line]table[x = n_cols, y = flop_rate_svds]\flopdata;
        \addplot[scur_line]table[x = n_cols, y = flop_rate_LUPP_sketch_cur]\flopdata;
        \addplot[icur_line]table[x = n_cols, y = flop_rate_LUPP_iterative_cur]\flopdata;
    \end{axis}

\end{tikzpicture}
    \caption{This plot shows the change in median relative error, median runtime, median number of FLOPs, and median FLOP rate for \svds, \icurl, and \curs{} with approximations of varying rank for a \textsc{G7Jac200} matrix.}
    \label{fig:fixed_rank_g7jac}
\end{figure}
\begin{figure}[h]
    \centering
    \begin{tikzpicture}[scale = \boxplotscale]
    \newcommand{\matname}{bayer01}
    \newcommand{\Matname}{\textsc{Bayer01}}
    \pgfplotstableread[col sep=comma]{./csvs/same_rank/\matname_250_accuracy.csv}\accdata
    \pgfplotstableread[col sep=comma]{./csvs/same_rank/\matname_250_time.csv}\timedata
    \pgfplotstableread[col sep=comma]{./csvs/same_rank/\matname_250_flops.csv}\flopdata
  \begin{axis}[
        name = axis1,
           ylabel = {Relative Frobenius Error},
        xlabel = {Rank of Approximation},
        title = {Accuracy for \Matname},
        ymode = log,
        ymajorgrids,
        legend pos = north east,
        scale only axis,
        xlabel near ticks,
        ylabel near ticks,
        ]
        \addplot[svds_line]table[x = n_cols, y = svds]\accdata;
        \addplot[scur_line]table[x = n_cols, y = LUPP_sketch_cur]\accdata;
        \addplot[icur_line]table[x = n_cols, y = LUPP_iterative_cur]\accdata;
        \legend{\svds, \curs, \icurl};
    \end{axis}
    
    \begin{axis}[
        name = axis2,
        at={(axis1.outer north east)},
        anchor=outer north west,
        xshift = 0.5em,
        ylabel = {Time (Seconds)},
        xlabel = {Rank of Approximation},
        title = {Time for \Matname},
        ymode = log,
        ymajorgrids,
        legend pos = south east,
        scale only axis,
        xlabel near ticks,
        ylabel near ticks,
        ]
        \addplot[svds_line]table[x = n_cols, y = svds]\timedata;
        \addplot[scur_line]table[x = n_cols, y = LUPP_sketch_cur]\timedata;
        \addplot[icur_line]table[x = n_cols, y = LUPP_iterative_cur]\timedata;
    \end{axis}
    \begin{axis}[
        name = axis3,
        at={(axis1.outer south west)},
        anchor=outer north west,
        yshift = -0.5em,
        ylabel = {Floating Point Operations},
        xlabel = {Rank of Approximation},
        title = {Total FLOPs for \Matname},
        ymode = log,
        ymajorgrids,
        legend pos = north east,
        scale only axis,
        xlabel near ticks,
        ylabel near ticks,
        ]
        \addplot[svds_line]table[x = n_cols, y = flops_svds]\flopdata;
        \addplot[scur_line]table[x = n_cols, y = flops_LUPP_sketch_cur]\flopdata;
        \addplot[icur_line]table[x = n_cols, y = flops_LUPP_iterative_cur]\flopdata;
    \end{axis}
    \begin{axis}[
        name = axis4,
        at={(axis2.outer south west)},
        anchor=outer north west,
        yshift = -0.5em,
        ylabel = {FLOP Rate},
        xlabel = {Rank of Approximation},
        title = {FLOP Rate for \Matname},
        ymode = log,
        ymin = 2e11,
        ymax = 5e12,
        ymajorgrids,
        legend pos = north east,
        scale only axis,
        xlabel near ticks,
        ylabel near ticks,
        ]
        \addplot[svds_line]table[x = n_cols, y = flop_rate_svds]\flopdata;
        \addplot[scur_line]table[x = n_cols, y = flop_rate_LUPP_sketch_cur]\flopdata;
        \addplot[icur_line]table[x = n_cols, y = flop_rate_LUPP_iterative_cur]\flopdata;
    \end{axis}

\end{tikzpicture}
    \caption{This plot shows the change in median relative error, median runtime, median number of FLOPs, and median FLOP rate for \svds, \icurl, and \curs{} with approximations of varying rank for a \textsc{Bayer01} matrix.}
    \label{fig:fixed_rank_bayer01}
\end{figure}

\subsection{Selection Methods}
It is clear that IterativeCUR with LUPP index selection is highly effective, it is next useful to consider whether using different index selection methods will have an impact on the performance. For this experiment, we consider the approximation quality of IterativeCUR when indices are selected with QRCP \cite{martinsson2017householder}, LUPP \cite{dong2023simpler}, and Osinsky\footnote{The Osinsky implementation selects columns with Osinsky and rows with QRCP.} \cite{osinsky2025close} for a \textsc{Lehmer} matrix with $1,000$ rows and columns from MATLAB's gallery and a  \textsc{Low-Rank PD}\footnote{For the diagonal matrix we take only the first 100 entries from the vector defined in \cref{subsec:fixed_rank}. Additionally, instead of rank $2,000$ we use a rank of $100$ for the two Gaussian matrices.} matrix with $1,000$ rows and columns. We consider each method's approximation performance with block size $50$ at ranks $\{150, 350, 550, 750,950\}$. For a baseline, we also consider the approximation performance of a truncated SVD and SVDSketch approximations. As we did not have an optimized implementation of Osinsky's selection approach, we did not consider the runtime, although it should be noted that Osinsky's is substantially slower than the alternatives.

We display the results of the experiment in \cref{fig:rand_svd_method_comp}. In both experiments, the LUPP and QRCP versions of IterativeCUR perform the same. The Osinsky method outperforms the other two selection approaches for the \textsc{Low-Rank PD} matrix, but is outperformed by the LUPP and QRCP approaches in the \textsc{Lehmer} matrix. Although initially it may be surprising that the best index selection approach does not always result in the best IterativeCUR approximation, one can make sense of this behavior by thinking about the conscientiously greedy interpretation of IterativeCUR discussed in \cref{sec:the}. Under this interpretation, at each iteration sketching the residual reveals the approximate dominant $b$-dimensional subspace of the residual. If this subspace contains enough information to make effective global index selections, then choosing the best possible indices will work exceptionally well. However, if the subspace does not contain enough information to make good global selections, then choosing the best indices for this $b$-dimensional subspace could mean poor global index selection because the index selection is more aggressive than the information quality warrants. This seems to be what occurs in these experiments, although it is very computationally expensive to truly verify whether the $b$-dimensional space contains enough information for effective global index selection in the \textsc{Low-Rank PD} case but not in the \textsc{Lehmer} case. Although, we can hypothesize this to be the case owing to \textsc{Lehmer} having a slower spectral decay than \textsc{Low-Rank PD}. What is remarkable is that even with this potentially greedy behavior in the \textsc{Lehmer} case, the worst-case approximation of the three methods is not substantially worse than the other two. This lends credence to the conscientiously greedy interpretation of IterativeCUR where making selections using the residual limits the downside risk that one would observe from a wholly greedy approach.

Overall, when one considers that in general LUPP is about twice as fast as QRCP and about three times as fast as Osinsky, it seems clear that the best selection approach to choose in most scenarios is LUPP. If one desires greater accuracy, Osinsky may result in better approximations and may be worth trying; however, as a default, we suggest the use of sketched LUPP in IterativeCUR.

\begin{figure}[h]
    \centering
    \begin{tikzpicture}[scale = \boxplotscale]
    \newcommand{\matname}{lehmer}
    \newcommand{\Matname}{Lehmer }
    \pgfplotstableread[col sep=comma]{./csvs/selection_methods/\matname_1000_50_med.csv}\accdata
    
    \begin{axis}[
        name = axis1,
        ylabel = {$\frac{\|A - \hat{A}_k\|_F}{\|A\|_F}$},
        xlabel = {Rank of Approximation},
        title = {Index Selection for \textsc{\Matname}},
        ymode = log,
        ymajorgrids,
        legend pos = south west,
        scale only axis,
        xlabel near ticks,
        ylabel near ticks,
        ]
        \addplot[svd_line]table[x = rank, y = svd]\accdata;
        \addplot[svds_line]table[x = rank, y = svds]\accdata;
        \addplot[qr_line]table[x = rank, y = qr]\accdata;
        \addplot[icur_line]table[x = rank, y = lupp]\accdata;
        \addplot[os_line]table[x = rank, y = os]\accdata;
        \legend{SVD, \svds, QRCP, LUPP, Osinsky};
    \end{axis}
\end{tikzpicture}
    \begin{tikzpicture}[scale = \boxplotscale]
    \newcommand{\matname}{Low_Rank_PD}
    \newcommand{\Matname}{\textsc{Low Rank PD} }
    \pgfplotstableread[col sep=comma]{./csvs/selection_methods/\matname_1000_50_med.csv}\accdata
    
    \begin{axis}[
        name = axis1,
        ylabel = {$\frac{\|A - \hat{A}_k\|_F}{\|A\|_F}$},
        xlabel = {Rank of Approximation},
        title = {Index Selection for \Matname},
        ymode = log,
        ymajorgrids,
        legend pos = south west,
        scale only axis,
        xlabel near ticks,
        ylabel near ticks,
        ]
        \addplot[svd_line]table[x = rank, y = svd]\accdata;
        \addplot[svds_line]table[x = rank, y = svds]\accdata;
        \addplot[qr_line]table[x = rank, y = qr]\accdata;
        \addplot[icur_line]table[x = rank, y = lupp]\accdata;
        \addplot[os_line]table[x = rank, y = os]\accdata;
        \legend{SVD, \svds, QRCP, LUPP, Osinsky};
    \end{axis}
\end{tikzpicture}
    \caption{The performance of IterativeCUR using \cite{osinsky2025close} (Osinsky), QR with column pivoting (QRCP), and LU with partial pivoting (LUPP) as the column and row selection strategies. }
    \label{fig:rand_svd_method_comp}
\end{figure}

\subsection{Block Size}
In addition to index selection strategies, another important aspect of IterativeCUR is the effect that block size has on runtime and accuracy. For this experiment, we compare IterativeCUR with LUPP index selection at block sizes 5, 50, 100, and 500 on the \textsc{Bayer01} and \textsc{Low-Rank PD} matrices from \cref{tab:fixed-rank} at ranks $\{500,1000,1500,2000,2500\}$.

\begin{figure}[h]
    \centering
    \begin{tikzpicture}[scale = \boxplotscale]
    \newcommand{\matname}{bayer01}
    \newcommand{\Matname}{\textsc{Bayer01} }
    \pgfplotstableread[col sep=comma]{./csvs/block_size/\matname_accuracy.csv}\accdata
    \pgfplotstableread[col sep=comma]{./csvs/block_size/\matname_time.csv}\timedata
    \begin{axis}[
        name = axis1,
        ylabel = {$\frac{\|A - \hat{A}_k\|_F}{\|A\|_F}$},
        xlabel = {Rank of Approximation},
        title = {Accuracy for \Matname Across Block Sizes},
        ymajorgrids,
        legend pos = north east,
        scale only axis,
        xlabel near ticks,
        ylabel near ticks,
        ]
        \addplot[bs5_line]table[x = n_cols, y = bs_5]\accdata;
        \addplot[bs50_line]table[x = n_cols, y = bs_50]\accdata;
        \addplot[bs100_line]table[x = n_cols, y = bs_100]\accdata;
        \addplot[bs500_line]table[x = n_cols, y = bs_500]\accdata;
        \legend{5, 50, 100, 500};
    \end{axis}
    \begin{axis}[
        name = axis2,
        at={(axis1.outer north east)},
        anchor=outer north west,
        ylabel = {Time (Seconds)},
        xlabel = {Rank of Approximation},
        title = {Time for \Matname Across Block Sizes},
        ymajorgrids,
        legend pos = south east,
        scale only axis,
        xlabel near ticks,
        ylabel near ticks,
        ]
        \addplot[bs5_line]table[x = n_cols, y = bs_5]\timedata;
        \addplot[bs50_line]table[x = n_cols, y = bs_50]\timedata;
        \addplot[bs100_line]table[x = n_cols, y = bs_100]\timedata;
        \addplot[bs500_line]table[x = n_cols, y = bs_500]\timedata;
    \end{axis}
\end{tikzpicture}
    \caption{Comparison of time and accuracy of IterativeCUR for block sizes 5, 50, 100, and 500 for a \textsc{Bayer01} matrix. The left plot shows median relative error over five runs while the right plot shows median runtime over five runs.}
    \label{fig:bayer01_bs}
\end{figure}
\begin{figure}[h]
    \centering
    \begin{tikzpicture}[scale = \boxplotscale]
    \newcommand{\matname}{low_Rank_PD}
    \newcommand{\Matname}{\textsc{Low-Rank PD} }
    \pgfplotstableread[col sep=comma]{./csvs/block_size/\matname_accuracy.csv}\accdata
    \pgfplotstableread[col sep=comma]{./csvs/block_size/\matname_time.csv}\timedata
    
   \begin{axis}[
        name = axis1,
        ylabel = {$\frac{\|A - \hat{A}_k\|_F}{\|A\|_F}$},
        xlabel = {Rank of Approximation},
        title = {Accuracy for \Matname Across Block Sizes},
        ymajorgrids,
        legend pos = north east,
        scale only axis,
        xlabel near ticks,
        ylabel near ticks,
        ]
        \addplot[bs5_line]table[x = n_cols, y = bs_5]\accdata;
        \addplot[bs50_line]table[x = n_cols, y = bs_50]\accdata;
        \addplot[bs100_line]table[x = n_cols, y = bs_100]\accdata;
        \addplot[bs500_line]table[x = n_cols, y = bs_500]\accdata;
        \legend{5, 50, 100, 500};
    \end{axis}
    \begin{axis}[
        name = axis2,
        at={(axis1.outer north east)},
        anchor=outer north west,
        ylabel = {Time (Seconds)},
        xlabel = {Rank of Approximation},
        title = {Time for \Matname Across Block Sizes},
        ymajorgrids,
        legend pos = south east,
        scale only axis,
        xlabel near ticks,
        ylabel near ticks,
        ]
        \addplot[bs5_line]table[x = n_cols, y = bs_5]\timedata;
        \addplot[bs50_line]table[x = n_cols, y = bs_50]\timedata;
        \addplot[bs100_line]table[x = n_cols, y = bs_100]\timedata;
        \addplot[bs500_line]table[x = n_cols, y = bs_500]\timedata;
    \end{axis}
\end{tikzpicture}
    \caption{Comparison of time and accuracy of IterativeCUR for block sizes 5, 50, 100, and 500 for a \textsc{Low-Rank PD} matrix. The left plot shows median relative error over five runs while the right plot shows median runtime over five runs.}
    \label{fig:low_rank_pd_bs}
\end{figure}

The results of these experiments are shown in \cref{fig:bayer01_bs} and \cref{fig:low_rank_pd_bs}. By first looking at the runtime we can see that for the sparse \textsc{Bayer01} matrix the choice of block size does not seem to have a substantial impact on runtime. When we consider the dense \textsc{Low-Rank PD} matrix, decreasing block size leads to slower performance. This makes sense because provided the matrix is not too big, large block sizes allow LAPACK to better exploit efficient blocked operations. 

In terms of accuracy, we see that, in the \textsc{Bayer01} matrix case, larger block sizes offer very small improvements in accuracy. Alternatively, for the \textsc{Low-Rank PD} matrix a smaller block size seems to lead to more accurate solutions, with a block size of 5 having about a 10 times lower median relative error than a block size of 500 and half the error as the block size of 50. These results align with the conscientiously greedy interpretation of IterativeCUR where smaller block sizes would be expected to produce approximations that are more accurate than the approximations produced by larger block sizes. 

Overall, this experiment suggests that the block size should be chosen based on the user's balancing of potential computational accelerations from blocking operations and the user's intuition on the rank of the matrix. When the user suspects that the matrix has a rank less than 100, then potentially choosing a block size as small as 5 may be sufficient. However, when the user expects the rank to be in the 1000s, moderate block sizes, around 200, should be chosen to best capitalize on the benefits of blocked LAPACK operations. 

\subsection{Large-Scale Comparison}
We finish our experiment section by comparing the methods on a sparse 18,571,154 by 18,571,154 \textsc{mawi\_201512012345} matrix from SuiteSparse \cite{davis2011university}. For this experiment, we approximate the matrix at fixed ranks $\{100, 200, 450, 900, 1800, 3600\}$ on the same computer as all previous experiments. For each trial, only one run is performed because of the large runtimes for SVDSketch and s-LUPP on high ranks. For IterativeCUR and SVDsketch, we perform the experiments using a block size of 50.

From the results of the experiments in \cref{fig:large_exp}, we can see that all three methods return essentially the same approximation for up to rank 900. Then, beyond rank 900, the memory requirements for the orthogonal matrix in SVDSketch and the sketch in s-LUPP exceed the resources available on the compute node, leading to substantial slowdowns or a halt (both methods were killed after 12 hours of runtime). IterativeCUR on the other hand continues to return accurate approximations without substantial computational costs, as $C$ and $R$ inherit the sparsity of $A$. This experiment demonstrates how the small sketch and sparsity preserving CUR approximation offered by IterativeCUR make it an attractive method for low-rank approximations of large sparse matrices.

\begin{figure}
    \centering
    \begin{tikzpicture}[scale = \boxplotscale]
    \newcommand{\matname}{mark3}
    \newcommand{\Matname}{\textsc{Mawai}}
    \pgfplotstableread[col sep=comma]{./csvs/same_rank/big_mawi.csv}\accdata
  \begin{axis}[
        name = axis1,
      ylabel = {Relative Frobenius Error},
        xlabel = {Rank of Approximation},
        title = {Accuracy for \Matname},
        ymode = log,
        xmode = log,
        ymajorgrids,
        legend pos = north east,
        scale only axis,
        xlabel near ticks,
        ylabel near ticks,
        unbounded coords=jump,
        ]
        \addplot[svds_line]table[x = rank, y =  res_svd]\accdata;
        \addplot[scur_line]table[x = rank, y =  res_slu]\accdata;
        \addplot[icur_line]table[x = rank, y = res_icur]\accdata;
        \legend{\svds, \curs, \icurl};
    \end{axis}
    
    \begin{axis}[
        name = axis2,
        at={(axis1.outer north east)},
        anchor=outer north west,
        xshift = 0.5em,
        ylabel = {Time (Seconds)},
        xlabel = {Rank of Approximation},
        title = {Time for \Matname},
        ymode = log,
        xmode = log,
        ymajorgrids,
        legend pos = south east,
        scale only axis,
        xlabel near ticks,
        ylabel near ticks,
        unbounded coords=jump,
        ]
        \addplot[svds_line]table[x = rank, y = time_svd]\accdata;
        \addplot[scur_line]table[x = rank, y = time_slu]\accdata;
        \addplot[icur_line]table[x = rank, y = time_icur]\accdata;
    \end{axis}
    \begin{axis}[
        name = axis3,
        at={(axis1.outer south west)},
        anchor=outer north west,
        yshift = -0.5em,
        ylabel = {Floating Point Operations},
        xlabel = {Rank of Approximation},
        title = {Total FLOPs for \Matname},
        ymode = log,
        xmode = log,
        ymajorgrids,
        legend pos = north east,
        scale only axis,
        xlabel near ticks,
        ylabel near ticks,
        unbounded coords=jump,
        ]
        \addplot[svds_line]table[x = n_cols, y = flops_svd]\accdata;
        \addplot[scur_line]table[x = n_cols, y = flops_slu]\accdata;
        \addplot[icur_line]table[x = n_cols, y = flops_icur]\accdata;
    \end{axis}
    \begin{axis}[
        name = axis4,
        at={(axis2.outer south west)},
        anchor=outer north west,
        yshift = -0.5em,
        ylabel = {FLOP Rate},
        xlabel = {Rank of Approximation},
        title = {FLOP Rate for \Matname},
        ymode = log,
        xmode = log,
        ymajorgrids,
        legend pos = north east,
        scale only axis,
        xlabel near ticks,
        ylabel near ticks,
        unbounded coords=jump,
        ]
        \addplot[svds_line]table[x = n_cols, y = flop_rate_svd]\accdata;
        \addplot[scur_line]table[x = n_cols, y = flop_rate_slu]\accdata;
        \addplot[icur_line]table[x = n_cols, y = flop_rate_icur]\accdata;
    \end{axis}
\end{tikzpicture}
    \caption{This plot shows the change in relative error, runtime,  number of FLOPs, and  FLOP rate for \svds, \icurl, and \curs{} with approximations of varying rank for a sparse 18,571,154 by 18,571,154 \textsc{mawi\_201512012345} matrix. Runs exceeding 12 hours were terminated.}
    \label{fig:large_exp}
\end{figure}

\section{Conclusion}\label{sec:con}
IterativeCUR is an efficient rank-adaptive method for computing CUR approximations. Because IterativeCUR's reuse of the sketch allows it to form high-quality approximations of matrices with only a small number of matrix-vector multiplications, it is exceptionally well suited for forming approximations to large matrices. Additionally, by incorporating risk-aware stopping criteria based on a posteriori estimators, IterativeCUR provides high-probability guarantees of approximation quality that a user can adjust to align with their confidence in their chosen stopping threshold. Further, while previous algorithms such as SVDSketch could only guarantee a relative accuracy of $\sqrt{\epsilon_{\rm mach}}$, IterativeCUR can easily handle tolerances very close to $\epsilon_{\rm mach}$. (This is achieved by directly estimating the norm of the residual, rather than relying on an indirect formula that expresses the error as a difference of two squared norms.)
In experiments, we have shown that IterativeCUR is highly competitive, matching and at times exceeding a state-of-the-art CUR method such as sketched LUPP, both in terms of accuracy and performance. Despite IterativeCUR's strong practical performance, work remains to be done to close the gap between the current conservative theoretical bounds and the observed results. 





\bibliographystyle{siamplain}
\bibliography{bib/references}
\end{document}